\documentclass[11pt]{amsart}
\usepackage{txfonts}
\usepackage{amsfonts}
\usepackage{amssymb}
\usepackage{amsmath}
\usepackage{mathrsfs}
\usepackage{mathtools}
\usepackage{breqn}
\usepackage[pdftex]{hyperref}

\title[Global Existence of Strong Solutions]{Global Existence of Strong Solutions to the Kinetic Cucker--Smale Model Coupled with the Stokes Equations}

\author[C. Jin]{Chunyin Jin}
\address[Chunyin Jin]{\newline College of Science,\newline China Agricultural University,\newline Beijing 100083, P. R. China.}
\email{jinchunyin@163.com}

\newtheorem{theorem}{Theorem}[section]
\newtheorem{definition}{Definition}[section]
\newtheorem{lemma}{Lemma}[section]
\newtheorem{proposition}{Proposition}[section]
\newtheorem{remark}{Remark}[section]
\newcommand{\bbr}{\mathbb R}
\newcommand{\bbn}{\mathbb N}

\newcommand{\bx}{\mbox{\boldmath $x$}}
\newcommand{\by}{\mbox{\boldmath $y$}}
\newcommand{\bu}{\mbox{\boldmath $u$}}
\newcommand{\bv}{\mbox{\boldmath $v$}}

\newcommand{\bj}{\mbox{\boldmath $j$}}

\begin{document}

\date{\today}



\begin{abstract}
In this paper, we investigate existence of global-in-time strong solutions to the kinetic Cucker--Smale model coupled with the Stokes equations in the whole space. By introducing a weighted Sobolev space and using space-time estimates for the linear non-stationary Stokes equations, we present a complete analysis on existence of global-in-time strong solutions to the coupled model, without any smallness requirements on initial data.
\end{abstract}

\maketitle \centerline{\date}
%
%
\section{Introduction}\label{sec-intro}
\setcounter{equation}{0}
In the present paper, we are concerned with global existence of strong solutions to the following kinetic Cucker--Smale model coupled with the Stokes equations in the whole space $\bbr^3$. For convenience,  $\nabla$ are abbreviated for $\nabla_{\bx}$, in someplace of the paper. The coupled kinetic-fluid model reads as
\begin{equation} \label{eq-cs-s}
     \begin{dcases}
         f_t + \bv \cdot \nabla_{\bx} f+ \nabla_{\bv} \cdot (L[f]f+(\bu-\bv)f)=0,\\
         \bu_t+\nabla P=\Delta \bu +\int_{\bbr^3}f(\bv-\bu)d\bv, \\
         \nabla \cdot \bu=0,
     \end{dcases}
\end{equation}
subject to the initial data
\begin{equation} \label{eq-sys-inidata}
  f|_{t=0}=f_0, \quad \bu|_{t=0}=\bu_0,
\end{equation}
with $\bu_0$ satisfying the compatibility condition $\nabla \cdot \bu_0=0$. Here $f(t,\bx, \bv)$ is the particle distribution function in phase space $(\bx, \bv)$ at the time $t$, $(\bx, \bv) \in \bbr^3\times \bbr^3$. $\bu$ and $P$ represent the fluid velocity and pressure, respectively. $L[f]$ is given by
\[
 L[f](t, \bx, \bv)=\int_{\bbr^{6}}\varphi(|\bx-\by|)f(t, \by, \bv^*)(\bv^*-\bv)d \by d \bv^*,
\]
where $\varphi(\cdot) \in C_{b}^{1}$ is a positive non-increasing function, standing for the interaction kernel. Without loss of generality, we postulate that
\[
 \max\{|\varphi|, |\varphi'|\} \le 1
\]
in the sequel.

Recently, collective behaviors of multi-agent systems have attracted much attention from researchers in diverse fields, including biology, physics, mathematics and control theory. People wish to understand mechanisms that lead to these phenomena, such as flocking and milling, by modeling, numerical simulation and mathematical analysis. In order to provide a justification for flocking, i.e., a multi-agent system reaches a consensus time-asymptotically, Cucker and Smale \cite{Cucker2007} put forward a system of ODEs, now entitled with their names, which resembles a Newton type $N$-body system. Moreover, they showed that flocking can be achieved under some conditions on initial data. Later, Ha--Liu \cite{Ha2009} presented a complete analysis on flocking using the Lyapunov functional approach, and further rigorously derived the kinetic Cucker--Smale model by taking the mean-field limit to the particle model. Then Carrillo et al. \cite{Carrillo2010} refined the results in \cite{Ha2009}, and provided an unconditional flocking theorem for measure-valued solutions to the kinetic Cucker--Smale model, with the same strength estimates valid as for the particle model. Along this line, Canizo et al. \cite{Canizo2011} contributed an elegant analysis on well-posedness of measure-valued solutions to some kinetic models of collective motion, by using the modern theory of optimal transport. A very recent research trend for the Cucker--Smale model from particle to kinetic and hydrodynamic descriptions has been launched. We refer readers to \cite{ha2014global}\cite{Ha2014}\cite{ha2015emergent}\cite{jin2015well}\cite{Jin} for studies related to the hydrodynamic Cucker--Smale model. If considering the Brownian effect in the modeling, the resulting model will contain a diffusive term. This kind of kinetic model is of the Fokker--Planck type, which admits an equilibrium. Duan \cite{duan2010kinetic} studied the stability and convergence rate of classical solutions to an equilibrium under small initial perturbations, by using the micro-macro decomposition. The interested readers can consult the review papers \cite{carrillo2010particle}\cite{choi2017emergent} for the state of the art in this research topic.

As in fact, particles are usually immersed in surrounding media, such as gas, water, and electromagnetic waves, etc. Taking into account the influence of ambient media, it is reasonable to incorporate these neglected effects in the modeling. Such coupled kinetic-fluid models have gained increasing interest due to their applications in biotechnology, medicine and sedimentation phenomena \cite{choi2017emergent}. The kinetic Cucker--Smale model coupled with the Stokes equations, incompressible Navier--Stokes equations, and isentropic compressible Navier--Stokes equations was introduced in \cite{Bae2012}\cite{bae2014asymptotic}\cite{bae2014global}\cite{bae2016global}, where existence of weak or strong solutions was investigated in spatial-periodic domain. However, the more physically relevant Cauchy problem was rarely touched, since the Poincar\'e inequality and the positive lower bound for the interaction kernel were crucially used in most previous analyses. Regretfully, these properties are difficult to guarantee in the whole space situation, which gives rise to some obstacles in the analysis of the Cauchy problem.

Recently, the author initiated the program to study the kinetic Cucker--Smale model and related coupled models with fluids in the whole space. In \cite{Jin2018}, Jin established the well-posedness of weak and strong solutions to the kinetic Cucker--Smale model by developing an unified framework, where weighted Sobolev
spaces were introduced to overcome the difficulty induced by unboundedness of the domain. Along this direction, then the author \cite{jin2019local} investigated the blowup criteria for strong solutions to the kinetic Cucker--Smale model coupled with the isentropic compressible Navier--Stokes equations in the whole space. It was shown that the integrability in time of the spatial $W^{1,\infty}$-norm on the fluid velocity controlled the blowup of strong solutions to the coupled model. Based on this observation, we are intended to explore the global-in-time strong solutions to the kinetic Cucker--Smale model coupled with the Stokes equations as the beginning. Before stating our theorem, we introduce the following weighted Sobolev space.
\begin{multline*}
 H_{\omega}^1(\bbr^3 \times \bbr^3):=\bigg\{h(\bx,\bv):\ h \in L_{\omega}^2(\bbr^3 \times \bbr^3), \\ \nabla_{\bx} h\in L_{\omega}^2(\bbr^3 \times \bbr^3),\ \nabla_{\bv}h \in L_{\omega}^2(\bbr^3 \times \bbr^3)\bigg\},
\end{multline*}
\[
 |h|_{H_{\omega}^1}^2:=|h|_{L_{\omega}^2}^2+|\nabla_{\bx}h|_{L_{\omega}^2}^2+|\nabla_{\bv}h|_{L_{\omega}^2}^2,
\]
where
\[
 |h|_{L_{\omega}^2}:=\left(\int_{\bbr^6}h^2(\bx,\bv)\omega(\bx,\bv)d\bx d\bv\right)^{\frac12},
\]
and
\[
 \omega(\bx,\bv):=(1+\bv^2)^{2\alpha+1}(1+\bx^2+\bv^2)^{3\gamma},\quad \alpha>1,\  \gamma>1.
\]
The weight $\omega(\bx,\bv)$ is introduced to overcome the difficulty arising from the coupling term. The reader will understand why we introduce such type of weight from the derivation of \eqref{eq-app-dif-est} in Sect. \ref{set-loc-exist}. Of course, the weight is not unique and even optimal, but it is convenient for our analysis. In this paper, we adopt the following simplified notations for homogeneous Sobolev Spaces.
\[
 D^1(\bbr^3):=\left\{u\in L^6(\bbr^3):\ \nabla u \in L^2(\bbr^3)\right\},
\]
\[
 D^2(\bbr^3):=\left\{u\in L_{loc}^1(\bbr^3):\ \nabla^2u \in L^2(\bbr^3)\right\},
\]
\[
 D^{2,p}(\bbr^3):=\left\{u\in L_{loc}^1(\bbr^3):\ \nabla^2u \in L^p(\bbr^3)\right\}, \quad 1\le p\le \infty.
\]
Next we give the definition of strong solutions to \eqref{eq-cs-s}-\eqref{eq-sys-inidata}.
\begin{definition} \label{def-stro}
Let $3<q\le 6, 0<T\le \infty$. $(f(t,\bx,\bv), \bu(t,\bx),\nabla P(t,\bx))$ is said to be a strong solution to \eqref{eq-cs-s}-\eqref{eq-sys-inidata}, if
\[
\begin{aligned}
  &f(t,\bx,\bv)\in C([0,T];H_{\omega}^1(\bbr^3 \times \bbr^3)),\\
  &\bu(t,\bx)\in C([0,T];H^{2}(\bbr^3))\cap L^2(0,T;D^{2,q}(\bbr^3)),\\
  &\bu_t(t,\bx)\in C([0,T];L^{2}(\bbr^3))\cap L^2(0,T;L^{q}(\bbr^3)),\\
  &\nabla P(t,\bx)\in C([0,T];L^{2}(\bbr^3))\cap L^2(0,T;L^{q}(\bbr^3)),
\end{aligned}
\]
and
\[
 \begin{gathered}
   \int_0^{\infty}\int_{\bbr^6}f \phi_t d\bx d\bv dt+\int_0^{\infty}\int_{\bbr^6}f \bv \cdot \nabla_{\bx}\phi d\bx d\bv dt\\+\int_0^{\infty}\int_{\bbr^6}\Big(fL[f]+f(\bu-\bv) \Big)\cdot \nabla_{\bv}\phi d\bx d\bv dt
   +\int_{\bbr^6}f_0 \phi(0) d\bx d\bv=0,
 \end{gathered}
\]
for all $\phi(t,\bx,\bv)\in C_0^{\infty}([0,T)\times \bbr^3 \times \bbr^3)$;
\[
 \begin{gathered}
   \int_0^{\infty}\int_{\bbr^3}\bu\cdot \boldsymbol{\psi}_t d\bx dt-\int_0^{\infty}\int_{\bbr^3}\bu \cdot \Delta \boldsymbol{\psi} d\bx dt-\int_0^{\infty}\int_{\bbr^6}f(\bv-\bu) \cdot \boldsymbol{\psi} d\bx d\bv dt\\
   +\int_{\bbr^3}\bu_0 \cdot\boldsymbol{\psi}(0) d\bx=0,
 \end{gathered}
\]
for all $\boldsymbol{\psi}(t,\bx)\in C_{0,\sigma}^{\infty}([0,T)\times \bbr^3)$, where
\[
 C_{0,\sigma}^{\infty}([0,T)\times \bbr^3):=\bigg\{\boldsymbol{\psi}(t,\bx): \ \boldsymbol{\psi}(t,\bx)\in C_0^{\infty}([0,T)\times \bbr^3), \ \nabla \cdot \boldsymbol{\psi}=0 \bigg\}.
\]
\end{definition}

Denote by $B(R_0)$ the ball centered at the origin with a radius $R_0$. Then the theorem in this paper can be stated as follows.
\begin{theorem} \label{thm-exist}
Let $0<R_0<\infty$. Assume the initial data $f_0(\bx,\bv)\ge 0$, $f_0(\bx,\bv) \in H_{\omega}^1(\bbr^3 \times \bbr^3)\cap L^{\infty}(\bbr^3 \times \bbr^3)$, and $\bu_0(\bx) \in H^{2}(\bbr^3)$,
with the $\bv$-support of $f_0(\bx,\bv)$ satisfying
\[
 \text{supp}_{\bv}f_0(\bx,\cdot)\subseteq B(R_0) \quad \text{for all $\bx \in \bbr^3$}.
\]
Then the Cauchy problem \eqref{eq-cs-s}-\eqref{eq-sys-inidata} admits a unique global-in-time strong solution in the sense of Definition \ref{def-stro}.
\end{theorem}

\begin{remark}
In terms of derivation of the kinetic Cucker--Smale model, it is reasonable to postulate boundedness of $\bv$-support of $f_0(\bx,\bv)$, since particle velocities are finite initially. Due to absence of the convection term in the Stokes equations, we do not require any smallness assumptions on the initial data.
\end{remark}

Even though the coupled model \eqref{eq-cs-s} has been studied in \cite{bae2016global}, however the current paper differs from \cite{bae2016global} mainly in two repects. First, our study is set in the whole space, instead of the spatial-periodic domain. We need to introduce some new weighted Sobolev space to overcome the difficulty caused by unboudedness of the domain. Second, the proof in \cite{bae2016global} essentially is based on regularity of weak solutions to the Stokes equations, while our proof lies in a priori estimates on the coupled system, together with the local existence analysis. The key to the proof is to obtain a priori estimate on $\int_0^T|\bu(t)|_{W^{1,\infty}}dt$ for all $0<T<\infty$. Using space-time estimates for the Stokes equations, cf. Proposition \ref{prop-s} in Sect. \ref{sec-preli}, and the Sobolev inequality, we can transform the estimate on $\int_0^T|\bu(t)|_{W^{1,\infty}}dt$ into estimates on $\rho(t,\bx):=\int_{\bbr^3}f(t,\bx,\bv)d\bv$ and $\bj(t,\bx):=\int_{\bbr^3}f(t,\bx,\bv)\bv d\bv$ in $L^{\infty}(0,T;L^q(\bbr^3)), 3<q\le6$. However, it is impossible to obtain estimates on $\rho(t,\bx)$ and $\bj(t,\bx)$ in $L^{\infty}(0,T;L^p(\bbr^3))$ for $p\ge2$, employing the traditional interpolation method. We circumvent this difficulty by means of the following strategy. Split the estimate on $\int_0^T|\bu(t)|_{W^{1,\infty}}dt$ into two steps. We first estimate $\int_0^T|\bu(t)|_{L^{\infty}}dt$. If this step is done, then we can obtain estimates on $f\langle \bv \rangle^k$ in $L^{\infty}(0,T;L^p(\bbr^3\times\bbr^3))$ for all $p, k \in (1,\infty)$, where $\langle \bv \rangle:=(1+\bv^2)^{\frac12}$. This yields estimates on $\rho(t,\bx)$ and $\bj(t,\bx)$ in $L^{\infty}(0,T;L^q(\bbr^3)), 3<q\le6$, using H\"older's inequality.
In order to obtain the estimate on $\int_0^T|\bu(t)|_{L^{\infty}}dt$, we still use the space-time estimates for the Stokes equations and the Sobolev inequality to transform this estimate into estimates on $\rho(t,\bx)$ and $\bj(t,\bx)$ in $L^{\infty}(0,T;L^2(\bbr^3))$. It is sufficient to estimate $f\langle \bv \rangle^3$ in $L^{\infty}(0,T;L^2(\bbr^3\times\bbr^3))$ to obtain these two estimates. Fortunately, the estimate on $f\langle \bv \rangle^3$ in $L^{\infty}(0,T;L^2(\bbr^3\times\bbr^3))$ can be achieved, by means of the $\langle \bv \rangle^6$-weighted energy estimate on $\eqref{eq-cs-s}_1$. With the estimate on $\int_0^T|\bu(t)|_{L^{\infty}}dt$ at hand, we further obtain the estimate on $\int_0^T|\bu(t)|_{W^{1,\infty}}dt$ by a bootstrap argument. The analysis in this section is completely new. Using the idea developed in this paper, it is an interesting problem to extend our result to the coupled model with the incompressible Navier--Stokes equations, under suitable conditions on initial data.

The rest of the paper is organized as follows. In Sect. \ref{sec-preli}, we present some preliminary results used in the subsequent analysis. In Sect. \ref{set-loc-exist}, we construct local-in-time strong solutions to the coupled model by iteration. In Sect. \ref{sec-apriori}, we derive some a priori estimates on the coupled model. Sect. \ref{sec-glob-exst} is devoted to the proof of our theorem.
\vskip 0.3cm
\noindent\textbf{Notation}. Throughout the paper, $C$ represents a general positive constant that may depend on $\varphi$, $\varphi'$, and the initial data. We write $C(\star)$ to emphasize that $C$ depends on $\star$. Both $C$ and $C(\star)$ may differ from line to line. The domain of a function norm is the whole space by default, for example, $|\bu(t,\cdot)|_{D^2}$ is short for $|\bu(t,\cdot)|_{D^2(\bbr^3)}$.

%
%
\section{Preliminary}\label{sec-preli}
\setcounter{equation}{0}
\subsection{The kinetic Cucker--Smale model}
When the number of particles is sufficiently large, it is not convenient to track dynamics of each particle using the ODEs model. Following the strategy from statistical physics, the kinetic Cucker--Smale model can be derived, by taking the mean-field limit to the particle Cucker--Smale model. Incorporating influences of surrounding media, the alignment term $fL[f]$ should be replaced by $fL[f]+f(\bu-\bv)$. Under some assumptions on the fluid velocity $\bu$, Jin \cite{jin2019local} recently provide a detailed analysis on $\eqref{eq-cs-s}_1$ in the weighted Sobolev space. Consider
\begin{equation} \label{eq-kine-cs}
     \begin{dcases}
         f_t + \bv \cdot \nabla_{\bx} f+ \nabla_{\bv} \cdot (L[f]f+(\bu-\bv)f)=0,\\
         f|_{t=0}=f_0(\bx,\bv),
     \end{dcases}
\end{equation}
for given $\bu(t,\bx)\in C([0,T];D^1\cap D^{2}(\bbr^3))\cap L^2(0,T;D^{2,q}(\bbr^3))$, $3<q\le 6$.
Define the bound of $\bv$-support of $f(t,\bx, \bv)$ at the time $t$ as
\[
  R(t):=\sup \Big\{|\bv|: \ (\bx,\bv) \in \text{supp} f(t,\cdot,\cdot)\Big\}.
\]
and
\[
 a(t,\bx):=\int_{\bbr^{6}} \varphi(|\bx-\by|)f(t, \by,\bv^*) d \by d \bv^*,
\]
\[
 \mathbf{b} (t,\bx):=\int_{\bbr^{6}} \varphi(|\bx-\by|)f(t, \by,\bv^*) \bv^* d \by d \bv^*.
\]
The following result is taken from  Proposition 2.1 in \cite{jin2019local}.
\begin{proposition}\label{prop-kine-cs-wp}
Let $0<R_0, T<\infty$. Assume $f_0(\bx,\bv) \ge 0$, $f_0(\bx,\bv) \in H_{\omega}^1(\bbr^3 \times \bbr^3)$, and $\text{supp}_{\bv}f_0(\bx,\cdot)\subseteq B(R_0)$ for all $\bx \in \bbr^3$. Given $\bu(t,\bx)\in C([0,T];D^1\cap D^{2}(\bbr^3))\cap L^2(0,T;D^{2,q}(\bbr^3))$, $3<q\le 6$, there exists a unique non-negative strong solution $f(t,\bx,\bv)\in C([0,T];H_{\omega}^1(\bbr^3 \times \bbr^3))$ to \eqref{eq-kine-cs}. Moreover,
\[
 \begin{aligned}
   (i)& \ R(t)\le R_0+\int_0^t(|\mathbf{b}(\tau)|_{L^{\infty}}+|\bu(\tau)|_{L^{\infty}})d\tau,\quad 0\le t \le T ;\\
   (ii)& \ |f(t)|_{H_{\omega}^1}\le |f_0|_{H_{\omega}^1}\exp \left( C \int_0^t \Big(1+R(\tau)+|\bu(\tau)|_{W^{1,\infty}} \Big) d \tau \right), \quad 0\le t \le T,
 \end{aligned}
\]
where $C:=C(\varphi, f_0)$.
\end{proposition}

\subsection{Linear non-stationary Stokes equations}
The fluid can be well approximated by a Stokes flow, when the velocity is very slow or the viscosity is very large. Given $\mathbf{g}(t,\bx)\in L^2(0,T;L^2\cap L^q(\bbr^3)), 3<q\le6$, consider the initial value problem to the following linear non-stationary Stokes equations.
\begin{equation} \label{eq-s}
     \begin{dcases}
         \bu_t+\nabla P=\Delta \bu +\mathbf g, \\
         \nabla \cdot \bu=0,\\
         \bu|_{t=0}=\bu_0,
     \end{dcases}
\end{equation}
with $\bu_0$ satisfying the compatibility condition $\nabla \cdot \bu_0=0$. The following result is summarized from Theorem 2.8 in \cite{giga1991abstract}, Theorem 1.5.2 and Lemma 1.6.2 in \cite{sohr2012navier}. It will be used in the construction of approximate solutions to \eqref{eq-cs-s}.
\begin{proposition}\label{prop-s}
Given $\mathbf{g}(t,\bx)\in L^2(0,T;L^2\cap L^q(\bbr^3)), 3<q\le6$, assume $\bu_0(\bx)\in H^2(\bbr^3)$. The Cauchy problem to the linear non-stationary Stokes equations \eqref{eq-s} admits a unique strong solution$(\bu,\nabla P)$ satisfying
\[
 \begin{aligned}
  &\bu(t,\bx)\in C([0,T];H^{2}(\bbr^3))\cap L^2(0,T;D^{2,q}(\bbr^3)),\\
  &\bu_t(t,\bx)\in C([0,T];L^{2}(\bbr^3))\cap L^2(0,T;L^{q}(\bbr^3)),\\
  &\nabla P(t,\bx)\in C([0,T];L^{2}(\bbr^3))\cap L^2(0,T;L^{q}(\bbr^3)),
 \end{aligned}
\]
for all $0<T\le \infty$. Moreover, there exists some constant $C$, independent of $T$, such that
\[
 |\bu_t|_{L^2(0,T;L^p)}+|\nabla^2\bu|_{L^2(0,T;L^p)}+|\nabla P|_{L^2(0,T;L^p)}\le C\Big(|\bu_0|_{H^2}+|\mathbf g|_{L^2(0,T;L^p)}\Big)
\]
for all $2\le p \le q$.
\end{proposition}
\subsection{The Helmholtz decomposition in $\bbr^3$}
It is well-known that any smooth vector field in $\bbr^3$ that falls off sufficiently fast at large distances can be uniquely decomposed as the sum of a divergence-free part and a gradient part. Denote by $L_{\sigma}^p(\bbr^3)$ the completion of $C_{0,\sigma}^{\infty}(\bbr^3):=\bigg\{\boldsymbol{\psi}(\bx):\ \boldsymbol{\psi}(\bx)\in C_{0}^{\infty}(\bbr^3), \nabla \cdot\boldsymbol{\psi}=0 \bigg\}$ in $L^p(\bbr^3)$, and
\[
 G^p(\bbr^3):=\bigg\{\mathbf h \in L^p(\bbr^3):\ \mathbf h=\nabla H \ \text{for some $H\in L_{loc}^p(\bbr^3)$}\bigg\}.
\]
Any vector field $\mathbf U(\bx)$ in $L^p(\bbr^3), 1<p<\infty$, can be uniquely decomposed as
\[
 \mathbf U(\bx)=\mathbf U_1(\bx)+\mathbf U_2(\bx), \quad \text{where $\mathbf U_1(\bx) \in L_{\sigma}^p(\bbr^3)$, and $\mathbf U_2(\bx) \in G^p(\bbr^3)$}.
\]
This decomposition is referred to as the Helmholtz decomposition. The corresponding Helmholtz projection $\mathcal P: L^p(\bbr^3) \mapsto L_{\sigma}^p(\bbr^3)$ is a bounded linear operator in $L^p(\bbr^3), 1<p<\infty$. This decomposition is widely used in fluid mechanics. We often project fluid equations on the space of divergence-free vector fields, to eliminate $\nabla P$. The Helmholtz decomposition in $\bbr^3$, cf. Remark III.1.1 and Theorem III.1.2 in \cite{galdi2011introduction}, is summarized as follows.
\begin{lemma}[The Helmholtz decomposition]\label{lem-heim-deco}
Given any vector field $\mathbf U(\bx)$ in $L^p(\bbr^3), 1<p<\infty$, there exists a unique $\Big(\mathbf U_1(\bx),\mathbf U_2(\bx)\Big)$ such that
\[
 \mathbf U(\bx)=\mathbf U_1(\bx)+\mathbf U_2(\bx), \quad \text{where $\mathbf U_1(\bx) \in L_{\sigma}^p(\bbr^3)$, and $\mathbf U_2(\bx) \in G^p(\bbr^3)$}.
\]
Moreover, $|\mathcal P\mathbf U |_{L^p}\le C|\mathbf U |_{L^p}$.
\end{lemma}
%
%
\section{Local Existence of Strong Solutions to the Coupled System}\label{set-loc-exist}
\setcounter{equation}{0}
In this section, we establish the local existence of strong solutions to the coupled system \eqref{eq-cs-s}-\eqref{eq-sys-inidata}. Our strategy is as follows. We first linearize the system and construct the approximate solutions by iteration. It is shown that there exists some $T_*>0$, depending only on the initial data and the model parameter, such that the approximate solutions are uniformly bounded in $[0, T_*]$. Then we prove that the approximate solution sequence is convergent in some lower-order regularity function spaces, and further show that the limit is the desired local strong solution. The result in this section is summarized as follows.
\begin{proposition} \label{prop-loc-exist}
Let $0<R_0<\infty, 3<q\le 6$. Assume the initial data $f_0(\bx,\bv)\ge 0$, $f_0(\bx,\bv) \in H_{\omega}^1(\bbr^3 \times \bbr^3)$, and $\bu_0(\bx) \in H^{2}(\bbr^3)$,
with the $\bv$-support of $f_0(\bx,\bv)$ satisfying
\[
 \text{supp}_{\bv}f_0(\bx,\cdot)\subseteq B(R_0) \quad \text{for all $\bx \in \bbr^3$}.
\]
Then there exists some $T_0>0$, depending only on the initial data and the model parameter, such that the Cauchy problem \eqref{eq-cs-s}-\eqref{eq-sys-inidata} admits a unique strong solution in $[0,T_0]$, satisfying
\[
\begin{aligned}
  &f(t,\bx,\bv)\in C([0,T_0];H_{\omega}^1(\bbr^3 \times \bbr^3)),\\
  &\bu(t,\bx)\in C([0,T_0];H^{2}(\bbr^3))\cap L^2(0,T_0;D^{2,q}(\bbr^3)),\\
  &\bu_t(t,\bx)\in C([0,T_0];L^{2}(\bbr^3))\cap L^2(0,T_0;L^{q}(\bbr^3)),\\
  &\nabla P(t,\bx)\in C([0,T_0];L^{2}(\bbr^3))\cap L^2(0,T_0;L^{q}(\bbr^3)).
\end{aligned}
\]
\end{proposition}
Next we use results in Sect. \ref{sec-preli} to finish the proof of Proposition \ref{prop-loc-exist}.
\vskip 0.3cm
\noindent \textit{Proof of Proposition \ref{prop-loc-exist}}. We first construct approximate solutions by iteration. Given $\bu^n(t, \bx) \in C([0,T];H^{2}(\bbr^3))\cap L^2(0,T;D^{2,q}(\bbr^3))$, $3<q\le 6$, with $\bu^n|_{t=0}=\bu_0$ in $H^{2}(\bbr^3)$, $(f^{n+1},\bu^{n+1},\nabla P^{n+1})$ is determined by
\begin{equation} \label{eq-cs-s-appro}
     \begin{dcases}
         f^{n+1}_t + \bv \cdot \nabla_{\bx} f^{n+1}+ \nabla_{\bv} \cdot (L[f^{n+1}]f^{n+1}+(\bu^{n}-\bv)f^{n+1})=0,\\
         \bu^{n+1}_t+\nabla P^{n+1}=\Delta \bu^{n+1}+\int_{\bbr^3}f^{n+1}(\bv-\bu^{n+1})d\bv,\\
         \nabla \cdot \bu^{n+1}=0,
     \end{dcases}
\end{equation}
subject to the initial data
\begin{equation} \label{eq-cs-s-approcau}
 f^{n+1}|_{t=0}=f_0,  \quad \bu^{n+1}|_{t=0}=\bu_0,
\end{equation}
with $\bu_0$ satisfying the compatibility condition $\nabla \cdot \bu_0=0$. From Proposition \ref{prop-kine-cs-wp} and \ref{prop-s}, we know $(f^{n+1},\bu^{n+1},\nabla P^{n+1})$ is well-defined. In the iteration procedure, $\bu^0$ is set by
\begin{equation} \label{eq-cs-s-approini}
     \begin{dcases}
       \bu^0_t=\Delta\bu^0,\\
       \bu^0|_{t=0}=\bu_0 \in H^2.
     \end{dcases}
\end{equation}
It is easy to see
\[
 \bu^0 \in C([0,\infty);H^{2})\cap L^2(0,\infty;D^{2,q}).
\]
Moreover, it holds that
\begin{equation} \label{eq-approiniest}
 \sup_{0\le t \le \infty}|\bu^0(t)|_{H^2}^2+ \int_0^\infty \Big(|\bu^0_t(t)|_{H^1}^2+|\bu^0(t)|_{D^{2,q}}^2\Big)dt \le C|\bu_0|_{H^2}^2.
\end{equation}
\vskip 0.3cm
\noindent \textbf{Uniform Bound on Approximate Solutions}
\vskip 0.3cm
Define
\[
 C_0:=C\Big(1+|f_0|_{H^1_{\omega}}^4\Big)\Big(1+|\bu_0|_{H^2}^2\Big).
\]
Suppose that there exists $T_* \in (0,T]$, to be determined later, such that
\begin{equation}\label{eq-indu-assup}
 \sup_{0\le t \le T_*}|\bu^n(t)|_{H^2}+\int_0^{T_*} \Big(|\bu^n_t(t)|_{H^1}^2+|\bu^n(t)|_{D^{2,q}}^2\Big) dt \le C_0, \quad n\in \bbn,
\end{equation}
Next we prove by induction that \eqref{eq-indu-assup} holds for all  $n\in \bbn$.
Using the induction hypothesis \eqref{eq-indu-assup} and taking $T_1:=T_1(\varphi,f_0,R_0,C_0)$ suitably small, we infer from Proposition \ref{prop-kine-cs-wp} that
\begin{equation}\label{eq-appro-f-norm}
    \sup_{0\le t \le T_1}\left|f^{n+1}(t)\right|_{H^1_{\omega}}\le 2\left|f_0\right|_{H^1_{\omega}}.
\end{equation}
Multiplying $\eqref{eq-cs-s-appro}_2$ by $\bu^{n+1}$ and integrating the resulting equation over $\bbr^3$, we have
\begin{equation}\label{eq-appro-dif-usqua}
  \begin{aligned}
    &\frac12 \frac{d}{dt}\left|\bu^{n+1}\right|_{L^2}^2 +\left|\nabla \bu^{n+1}\right|_{L^2}^2\\
    =&\int_{\bbr^3}\int_{\bbr^3}f^{n+1}(\bv-\bu^{n+1}) \cdot \bu^{n+1}d\bv d\bx\\
    \le&C\left|\int_{\bbr^3} f^{n+1} \bv d\bv\right|_{L^{\frac65}} \left|\nabla \bu^{n+1}\right|_{L^2}\\
    \le&\frac12 \left|\nabla \bu^{n+1}\right|_{L^2}^2+C\left|f^{n+1}\right|_{L_{\omega}^2}^2,
  \end{aligned}
\end{equation}
where we have used the inequality
\begin{equation*}
  \begin{aligned}
    \left|\int_{\bbr^3} f^{n+1} \bv d\bv\right|_{L^{\frac65}}\le &\left|\int_{\bbr^3}f^{n+1} \bv d\bv\right|_{L^1}^{\frac23}\left|\int_{\bbr^3}f^{n+1} \bv d\bv\right|_{L^2}^{\frac13}\\
    \le&\frac23 \left|\int_{\bbr^3}f^{n+1} \bv d\bv\right|_{L^1}+ \frac13 \left|\int_{\bbr^3} f^{n+1}\bv d\bv\right|_{L^2}\\
    \le &C\left|f^{n+1}\right|_{L_{\omega}^2}.
  \end{aligned}
\end{equation*}
Take $0<T_2 \le T_1$. Integrating \eqref{eq-appro-dif-usqua} over $[0,T_2]$ leads to
\begin{equation}\label{eq-appro-usqua}
  \sup_{0\le t \le T_2} \left|\bu^{n+1}(t)\right|_{L^2}^2 +\int_0^{T_2} \left|\nabla \bu^{n+1}(t)\right|_{L^2}^2 dt \le \left|\bu_0\right|_{L^2}^2 +C\left|f_0\right|_{H_{\omega}^1}^2 T_2.
\end{equation}
Multiplying $\eqref{eq-cs-s-appro}_2$ by $\bu^{n+1}_t$ and integrating the resulting equation over $\bbr^3$, we deduce that
\begin{equation}\label{eq-appro-dif-gradusqua}
  \begin{aligned}
    &\frac12 \frac{d}{dt}\left|\nabla \bu^{n+1}\right|_{L^2}^2 +\left|\bu^{n+1}_t\right|_{L^2}^2\\
    =&\int_{\bbr^3}\int_{\bbr^3}f^{n+1}(\bv-\bu^{n+1}) \cdot \bu^{n+1}_t d\bv d\bx\\
    \le&\left|\int_{\bbr^3} f^{n+1} \bv d\bv\right|_{L^2} \left|\bu^{n+1}_t\right|_{L^2}+C \left|\int_{\bbr^3} f^{n+1} d\bv\right|_{L^3} \left|\nabla \bu^{n+1}\right|_{L^2} \left|\bu^{n+1}_t\right|_{L^2}\\
    \le&\frac12 \left|\bu^{n+1}_t\right|_{L^2}^2+C\left|f^{n+1}\right|_{H_{\omega}^1}^2 \Big( 1+\left|\nabla \bu^{n+1}\right|_{L^2}^2 \Big),
  \end{aligned}
\end{equation}
where we have used the inequality
\begin{equation*}
  \begin{aligned}
    \left|\int_{\bbr^3}f^{n+1}  d\bv\right|_{L^3}=&\left|\int_{\bbr^3}f^{n+1}  d\bv\right|_{L^2}^{\frac12}\left|\int_{\bbr^3}f^{n+1} d\bv\right|_{L^6}^{\frac12}\\
    \le&\frac12 \left|\int_{\bbr^3}f^{n+1} d\bv\right|_{L^2}+ \frac12 \left|\int_{\bbr^3}\left|\nabla f^{n+1}\right| d\bv\right|_{L^2}\\
    \le &C\left|f^{n+1}\right|_{H_{\omega}^1}.
  \end{aligned}
\end{equation*}
Integrating \eqref{eq-appro-dif-gradusqua} over $[0,T_2]$ gives
\begin{equation}\label{eq-appro-gradusqua}
  \begin{gathered}
   \sup_{0\le t \le T_2} \left|\nabla \bu^{n+1}(t)\right|_{L^2}^2 +\int_0^{T_2} \left|\bu_t^{n+1}(t)\right|_{L^2}^2 dt\\ \le \left|\nabla \bu_0\right|_{L^2}^2 +C\left|f_0\right|_{H_{\omega}^1}^2 \Big(T_2 +\left|\bu_0\right|_{L^2}^2 +C\left|f_0\right|_{H_{\omega}^1}^2 T_2 \Big).
  \end{gathered}
\end{equation}
Differentiating $\eqref{eq-cs-s-appro}_2$ with respect to $t$, we infer that
\begin{equation}\label{eq-appro-difttu}
  \begin{gathered}
   \bu^{n+1}_{tt}+ \nabla P^{n+1}_t=\Delta \bu^{n+1}_t +\int_{\bbr^3}f^{n+1}_t(\bv-\bu^{n+1})d\bv\\ -\int_{\bbr^3}f^{n+1}d\bv \bu^{n+1}_t, \quad \text{in $\mathcal{D}'([0,T)\times \bbr^3)$}.
  \end{gathered}
\end{equation}
Take $\bu^{n+1}_t$ as the test function. It follows from \eqref{eq-appro-difttu} that
\begin{equation}\label{eq-approdif-gradusqua}
    \frac12 \frac{d}{dt}\left|\bu_t^{n+1}\right|_{L^2}^2 +\left|\nabla \bu^{n+1}_t\right|_{L^2}^2
    \le \int_{\bbr^3}\int_{\bbr^3}f_t^{n+1}(\bv-\bu^{n+1}) \cdot \bu^{n+1}_t d\bv d\bx,
\end{equation}
Using $\eqref{eq-cs-s-appro}_1$, we estimate the right-hand side of \eqref{eq-appro-dif-gradusqua} as follows.
\begin{equation}\label{eq-approdif-gradusqua-rh}
  \begin{aligned}
    &\int_{\bbr^3}\int_{\bbr^3}f_t^{n+1}(\bv-\bu^{n+1}) \cdot \bu^{n+1}_t d\bv d\bx\\
    =&-\int_{\bbr^3} \int_{\bbr^3} \Big[\bv \cdot \nabla_{\bx} f^{n+1}+ \nabla_{\bv} \cdot \big(L[f^{n+1}]f^{n+1}+(\bu^{n}-\bv)f^{n+1}\big)\Big] (\bv-\bu^{n+1}) \cdot \bu^{n+1}_t  d\bv d\bx\\
    =&\int_{\bbr^3} \int_{\bbr^3} f^{n+1} \bv \otimes (\bv-\bu^{n+1}) d\bv : \nabla \bu^{n+1}_t d\bx-\int_{\bbr^3} \int_{\bbr^3} f^{n+1} \bv d\bv \cdot \nabla \bu^{n+1} \cdot \bu^{n+1}_t d\bx\\
    &+\int_{\bbr^3} \int_{\bbr^3} \Big[f^{n+1}L[f^{n+1}]+f^{n+1}(\bu^{n}-\bv)\Big] d\bv \cdot \bu^{n+1}_t d\bx\\
    \le&C \left|\int_{\bbr^3} f^{n+1} \bv^2 d\bv\right|_{L^2} \left|\nabla \bu^{n+1}_t\right|_{L^2}+C \left|\int_{\bbr^3} f^{n+1}\bv d\bv\right|_{L^3} \left|\nabla \bu^{n+1}\right|_{L^2} \left|\nabla \bu^{n+1}_t\right|_{L^2}\\
    &+C\left|f^{n+1} \langle \bv \rangle \right|_{L^1} \left|\int_{\bbr^3} f^{n+1} \langle \bv \rangle d\bv\right|_{L^{\frac65}}\left|\nabla \bu^{n+1}_t\right|_{L^2}\\
     &+C\left(1+\left|\bu^n\right|_{L^\infty} \right) \left| \int_{\bbr^3} f^{n+1} \langle \bv \rangle d\bv\right|_{L^{\frac65}}\left|\nabla \bu^{n+1}_t\right|_{L^2}\\
     \le&\frac12 \left|\nabla \bu^{n+1}_t\right|_{L^2}^2+C\left|f^{n+1}\right|_{H_{\omega}^1}^4+ C\left|f^{n+1}\right|_{H_{\omega}^1}^2 \Big( 1+\left|\nabla \bu^{n+1}\right|_{L^2}^2 \Big)+C\left|f^{n+1}\right|_{H_{\omega}^1}^2 \left( 1+\left|\bu^n\right|_{L^\infty}^2 \right).
  \end{aligned}
\end{equation}
In the derivation of the last inequality in \eqref{eq-approdif-gradusqua-rh}, we have used the following inequalities.
\begin{equation*}
  \begin{aligned}
    \left|\int_{\bbr^3}f^{n+1} \bv^2 d\bv\right|_{L^2}=&\left|\int_{\bbr^3}f^{n+1} \bv^2 \langle \bv \rangle^{2\alpha}\langle \bv \rangle^{-2\alpha}d\bv\right|_{L^2}\\
    \le& \left|\langle \bv \rangle^{-2\alpha}\right|_{L^2}\left| f^{n+1} \bv^2 \langle \bv \rangle^{2\alpha}\right|_{L^2}\\
    \le& C\left|f^{n+1}\right|_{L_{\omega}^2};
  \end{aligned}
\end{equation*}
\begin{equation*}
  \begin{aligned}
    \left| f^{n+1} \langle \bv \rangle\right|_{L^1}=&\left|\int_{\bbr^3}\int_{\bbr^3}f^{n+1} \langle \bv \rangle(1+\bx^2+\bv^2)^{\frac{3\gamma}{2}}(1+\bx^2+\bv^2)^{-\frac{3\gamma}{2}}d\bx d\bv\right|\\
    \le& \left|(1+\bx^2+\bv^2)^{-\frac{3\gamma}{2}}\right|_{L^2}|f^{n+1}|_{L_{\omega}^2}\\
    \le& C\left|f^{n+1}\right|_{L_{\omega}^2};
  \end{aligned}
\end{equation*}
\begin{equation*}
  \begin{aligned}
    \left|\int_{\bbr^3}f^{n+1} \bv d\bv\right|_{L^3}=&\left|\int_{\bbr^3}f^{n+1} \bv d\bv\right|_{L^1}^{\frac15}\left|\int_{\bbr^3}f^{n+1} \bv d\bv\right|_{L^6}^{\frac45}\\
    \le&\frac15 \left|\int_{\bbr^3}f^{n+1} \bv d\bv\right|_{L^1}+ \frac45 \left|\int_{\bbr^3}|\nabla f^{n+1}| |\bv| d\bv\right|_{L^2}\\
    \le &C\left|f^{n+1}\right|_{H_{\omega}^1};\\[2mm]
    \left|\int_{\bbr^3} f^{n+1} \langle \bv \rangle d\bv\right|_{L^{\frac65}}\le &\left|\int_{\bbr^3}f^{n+1} \langle \bv \rangle d\bv\right|_{L^1}^{\frac23}\left|\int_{\bbr^3}f^{n+1} \langle \bv \rangle d\bv\right|_{L^2}^{\frac13}\\
    \le&\frac23 \left|\int_{\bbr^3}f^{n+1} \langle \bv \rangle d\bv\right|_{L^1}+ \frac13 \left|\int_{\bbr^3} f^{n+1}\langle \bv \rangle d\bv\right|_{L^2}\\
    \le &C\left|f^{n+1}\right|_{L_{\omega}^2}.
  \end{aligned}
\end{equation*}
Since $\bu^{n+1}_t \in C([0,T_1];L^2)$, we have
\begin{equation}\label{eq-appro-dift-usquainit}
  \begin{aligned}
    \left|\bu^{n+1}_t(0) \right|_{L^2}^2=&\left|\Delta \bu_0 +\mathcal P \int_{\bbr^3} f_0(\bv-\bu_0)d\bv \right|_{L^2}^2\\
     \le& C \left|\bu_0 \right|_{D^2}^2 +C\left|f_0 \right|_{L_{\omega}^2}^2 +C\left|f_0 \right|_{L_{\omega}^2}^2\left|\bu_0 \right|_{H^2}^2\\
     \le& C\left(1+\left|\bu_0 \right|_{H^2}^2\right) \left(1+\left|f_0 \right|_{H_{\omega}^1}^2\right).
  \end{aligned}
\end{equation}
Substituting \eqref{eq-approdif-gradusqua-rh} into \eqref{eq-approdif-gradusqua}, and integrating the resulting inequality over $[0,T_2]$ lead to
\begin{equation}\label{eq-approdift-gradusqua}
  \begin{aligned}
   &\sup_{0\le t \le T_2} \left|\bu_t^{n+1}(t)\right|_{L^2}^2 +\int_0^{T_2} \left|\nabla \bu_t^{n+1}(t)\right|_{L^2}^2 dt\\
   \le &\left| \bu_t^{n+1}(0)\right|_{L^2}^2 +C(C_0)\left|f_0\right|_{H_{\omega}^1}^2 T_2 +C\left|f_0\right|_{H_{\omega}^1}^4 T_2+ C\left|f_0\right|_{H_{\omega}^1}^2 \int_0^{T_2} \left|\nabla \bu^{n+1}(t)\right|_{L^2}^2 dt\\
   \le& C\left(1+\left|\bu_0 \right|_{H^2}^2\right) \left(1+\left|f_0 \right|_{H_{\omega}^1}^2\right) +C(C_0)\left(1+\left|f_0 \right|_{H_{\omega}^1}^4\right)T_2 \\&+C\left|f_0 \right|_{H_{\omega}^1}^2 \left(\left|\bu_0\right|_{L^2}^2 +C\left|f_0\right|_{H_{\omega}^1}^2 T_2 \right),
  \end{aligned}
\end{equation}
where we have used the induction assumption \eqref{eq-indu-assup}, \eqref{eq-appro-usqua},  and  \eqref{eq-appro-dift-usquainit}. Take $T_2:=T_2(f_0,C_0)$ suitably small. We know from \eqref{eq-appro-usqua}, \eqref{eq-appro-gradusqua}, and \eqref{eq-approdift-gradusqua} that
\begin{equation}\label{eq-appro-ugraduudiftsqua}
  \begin{aligned}
    &\sup_{0\le t \le T_2} \left|\bu^{n+1}(t)\right|_{L^2}^2 +\int_0^{T_2} \left|\bu_t^{n+1}(t)\right|_{L^2}^2 dt \le 1+ \left|\nabla \bu_0\right|_{L^2}^2;\\
    &\sup_{0\le t \le T_2} \left|\nabla \bu^{n+1}(t)\right|_{L^2}^2 +\int_0^{T_2} \left|\bu_t^{n+1}(t)\right|_{L^2}^2 dt
   \le C\Big(1+\left|\bu_0 \right|_{H^1}^2\Big) \Big(1+\left|f_0 \right|_{H_{\omega}^1}^2\Big);\\
    &\sup_{0\le t \le T_2} \left|\bu_t^{n+1}(t)\right|_{L^2}^2 +\int_0^{T_2} \left|\nabla \bu_t^{n+1}(t)\right|_{L^2}^2 dt
   \le C\Big(1+\left|\bu_0 \right|_{H^2}^2\Big) \Big(1+\left|f_0 \right|_{H_{\omega}^1}^2\Big).
  \end{aligned}
\end{equation}
Project $\eqref{eq-cs-s-appro}_2$ on the divergence-free field to eliminate the pressure term. We obtain
\begin{equation}\label{eq-appro-cs-s-df}
  -\Delta \bu^{n+1}=-\bu_t^{n+1}+\mathcal P \int_{\bbr^3} f^{n+1}(\bv-\bu^{n+1}) d\bv.
\end{equation}
From elliptic estimates on \eqref{eq-appro-cs-s-df}, we deduce that
\begin{equation}\label{eq-appro-gradtwousqua}
  \left| \bu^{n+1} \right|_{D^2}^2 \le C\left| \bu_t^{n+1} \right|_{L^2}^2+ C\left|f^{n+1}\right|_{H_{\omega}^1}^2 \Big( 1+\left|\nabla \bu^{n+1}\right|_{L^2}^2 \Big).
\end{equation}
By virtue of \eqref{eq-appro-f-norm} and $\eqref{eq-appro-ugraduudiftsqua}_2-\eqref{eq-appro-ugraduudiftsqua}_3$, we have
\begin{equation}\label{eq-appro-uggradsqua}
    \sup_{0\le t \le T_2} \left|\bu^{n+1}(t)\right|_{D^2}^2 \le C\Big(1+\left|\bu_0 \right|_{H^2}^2\Big) \Big(1+\left|f_0 \right|_{H_{\omega}^1}^4\Big).
\end{equation}
We employ the elliptic estimates on \eqref{eq-appro-cs-s-df} again to obtain
\begin{equation}\label{eq-appro-gradtwousixsqua}
  \left| \bu^{n+1} \right|_{D^{2,6}}^2 \le C\left|\nabla \bu_t^{n+1} \right|_{L^2}^2+ C\left|f^{n+1}\right|_{H_{\omega}^1}^2 \Big( 1+\left| \bu^{n+1}\right|_{H^2}^2 \Big),
\end{equation}
where we have used the Sobolev inequality
\begin{equation}\label{eq-linftsobolev}
  \left| \bu^{n+1} \right|_{L^\infty}\le  C\left| \bu^{n+1}\right|_{H^2} \quad \text{in $\bbr^3$}.
\end{equation}
Take $0<T_3\le T_2$ suitably small. Using \eqref{eq-appro-ugraduudiftsqua}, and \eqref{eq-appro-gradtwousqua}-\eqref{eq-appro-gradtwousixsqua}, we get by interpolation that
\begin{equation}\label{eq-appro-uggradqsqua}
  \begin{aligned}
    \int_0^{T_3} \left|\bu^{n+1}(t)\right|_{D^{2,q}}^2 dt \le& 2\int_0^{T_3} \left|\bu^{n+1}(t)\right|_{D^{2}}^2 dt+ 2\int_0^{T_3} \left|\bu^{n+1}(t)\right|_{D^{2,6}}^2 dt\\
   \le& C\Big(1+\left|\bu_0 \right|_{H^1}^2\Big) \Big(1+\left|f_0 \right|_{H_{\omega}^1}^2\Big).
  \end{aligned}
\end{equation}
Let $T_*:=\min\{T_1,T_2,T_3\}$. Adding \eqref{eq-appro-ugraduudiftsqua}, \eqref{eq-appro-uggradsqua} and \eqref{eq-appro-uggradqsqua} together, we obtain
\begin{equation}\label{eq-appro-induassunext}
    \sup_{0\le t \le T_*}|\bu^{n+1}(t)|_{H^2}+\int_0^{T_*}  \Big(|\bu^{n+1}_t(t)|_{H^1}^2+|\bu^{n+1}(t)|_{D^{2,q}}^2\Big) dt \le C_0.
\end{equation}
From \eqref{eq-approiniest}, we know $\bu^0(t,\bx)$ also satisfies \eqref{eq-indu-assup}. Thus, we conclude by induction that \eqref{eq-indu-assup} holds for all $n\in \bbn$.
\vskip 3mm
\noindent \textbf{Convergence of Approximate Solutions}
\vskip 3mm
Define
\[
 \overline{f}^{n+1}:=f^{n+1}-f^n,\quad \overline{\bu}^{n+1}:=\bu^{n+1}-\bu^n,\quad \overline{P}^{n+1}:=P^{n+1}-P^n.
\]
It follows from \eqref{eq-cs-s-appro}-\eqref{eq-cs-s-approcau} that
\begin{equation} \label{eq-cs-s-appro-dif}
     \begin{dcases}
      \overline{f}^{n+1}_t +\bv \cdot\nabla_{\bx} \overline{f}^{n+1} +\nabla_{\bv}\cdot\Big[L[f^{n+1}]\overline{f}^{n+1}+(\bu^n-\bv)\overline{f}^{n+1}\Big]\\
           \qquad+\nabla_{\bv}\cdot\Big[L[\overline{f}^{n+1}]f^n +f^n \overline{\bu}^n\Big]=0,\\
       \overline{\bu}^{n+1}_t  +\nabla\overline{P}^{n+1}
          =\Delta \overline{\bu}^{n+1}  -\int_{\bbr^3}f^{n}\overline{\bu}^{n+1}d\bv +\int_{\bbr^3}\overline{f}^{n+1}(\bv-\bu^{n+1})d\bv,\\
       \nabla\cdot\overline{\bu}^{n+1}=0,
     \end{dcases}
\end{equation}
and
\begin{equation} \label{eq-cs-s-approini-dif}
 \overline{f}^{n+1}|_{t=0}=0,  \quad \overline{\bu}^{n+1}|_{t=0}=0.
\end{equation}
Multiplying $\eqref{eq-cs-s-appro-dif}_2$ by $\overline{\bu}^{n+1}$ and integrating the resulting equation over $\bbr^3$, we deduce that
\begin{equation}\label{eq-app-udif-squa}
  \begin{aligned}
    &\frac12\frac{d}{dt}\left|\overline{\bu}^{n+1}\right|_{L^2}^2 +\left|\nabla\overline{\bu}^{n+1}\right|_{L^2}^2\\
    \le& \int_{\bbr^3}\int_{\bbr^3}\overline{f}^{n+1}(\bv-\bu^{n+1})d\bv\cdot\overline{\bu}^{n+1}d\bx\\
   \le&\left|\int_{\bbr^3}\overline{f}^{n+1}\bv d\bv\right|_{L^{\frac32}} \left|\overline{\bu}^{n+1}\right|_{L^3} +\left|\int_{\bbr^3}\overline{f}^{n+1} d\bv\right|_{L^{\frac32}} \left|\bu^{n+1}\right|_{L^6} \left|\overline{\bu}^{n+1}\right|_{L^6}\\
   \le& C\left|\int_{\bbr^3}\overline{f}^{n+1}\bv d\bv\right|_{L^{\frac32}} \bigg(\left|\overline{\bu}^{n+1}\right|_{L^2}+ \left|\nabla\overline{\bu}^{n+1}\right|_{L^2} \bigg)\\
   &+C\left|\int_{\bbr^3}\overline{f}^{n+1} d\bv\right|_{L^{\frac32}} \left|\nabla{\bu}^{n+1}\right|_{L^2} \left|\nabla\overline{\bu}^{n+1}\right|_{L^2}\\
   \le& \frac12 \left|\nabla\overline{\bu}^{n+1}\right|_{L^2}^2 +\frac12 \left|\overline{\bu}^{n+1}\right|_{L^2}^2 +C\bigg(1+ \left|\nabla{\bu}^{n+1}\right|_{L^2}^2 \bigg)\left|\int_{\bbr^3}\Big|\overline{f}^{n+1}\Big|\langle \bv \rangle d\bv\right|_{L^{\frac32}},
  \end{aligned}
\end{equation}
that is,
\begin{equation}\label{eq-app-udif-squashort}
  \frac{d}{dt}\left|\overline{\bu}^{n+1}\right|_{L^2}^2 +\left|\nabla\overline{\bu}^{n+1}\right|_{L^2}^2
   \le \left|\overline{\bu}^{n+1}\right|_{L^2}^2 \\+C\bigg(1+ \left|\nabla{\bu}^{n+1}\right|_{L^2}^2 \bigg)\left|\overline{f}^{n+1} \left(1+ \bv^2\right)^{\alpha}\right|_{L^{\frac32}},
\end{equation}
where we have used the following inequality
\[
\begin{aligned}
  &\left|\int_{\bbr^3}\Big|\overline{f}^{n+1}\Big|\langle \bv \rangle d\bv\right|_{L^{\frac32}}\\
  \le& \left(\int_{\bbr^6}\bigg[\Big|\overline{f}^{n+1}\Big| \left(1+ \bv^2\right)^{\alpha}\bigg]^{\frac32}d\bv d\bx \right)^{\frac23} \left(\int_{\bbr^3}(1+ \bv^2)^{\frac32-3\alpha}d\bv \right)^{\frac13}\\
  \le& C\left|\overline{f}^{n+1} \left(1+ \bv^2\right)^{\alpha}\right|_{L^{\frac32}}, \quad \alpha>1.
 \end{aligned}
\]
Define $\Lambda(\bv):= \left(1+ \bv^2\right)^{\alpha}$, $\ \alpha>1$. Multiplying $\eqref{eq-cs-s-appro-dif}_1$ by $\Lambda(\bv)$, we deduce that
\begin{equation}\label{eq-app-kcswt-dif}
  \begin{aligned}
    &\Big(\overline{f}^{n+1}\Lambda\Big)_t +\bv \cdot\nabla_{\bx} \Big(\overline{f}^{n+1}\Lambda\Big) +\nabla_{\bv}\cdot\Big[L[f^{n+1}]\overline{f}^{n+1}\Lambda+(\bu^n-\bv)\overline{f}^{n+1}\Lambda\Big]\\
    =&L[f^{n+1}]\cdot\nabla_{\bv}\Lambda\overline{f}^{n+1} +(\bu^n-\bv)\cdot\nabla_{\bv}\Lambda\overline{f}^{n+1}\\
      &-\bigg(\nabla_{\bv}\cdot L[\overline{f}^{n+1}]f^n +L[\overline{f}^{n+1}]\cdot\nabla_{\bv} f^n \bigg)\Lambda -\overline{\bu}^n \cdot\nabla_{\bv} f^n \Lambda.
  \end{aligned}
\end{equation}
Multiplying \eqref{eq-app-kcswt-dif} by $\frac32 \left|\overline{f}^{n+1}\Lambda\right|^{\frac12}\text{sgn}\overline{f}^{n+1}$ leads to
\begin{equation}\label{eq-app-kcswt-dif-frac}
  \begin{aligned}
    &\frac{\partial}{\partial t}\left|\overline{f}^{n+1}\Lambda\right|^{\frac32} +\bv \cdot\nabla_{\bx} \left|\overline{f}^{n+1}\Lambda\right|^{\frac32} +\nabla_{\bv}\cdot\bigg[L[f^{n+1}]\left|\overline{f}^{n+1}\Lambda\right|^{\frac32} +(\bu^n-\bv)\left|\overline{f}^{n+1}\Lambda\right|^{\frac32}\bigg]\\
    =&-\frac12 \nabla\cdot L[f^{n+1}]\left|\overline{f}^{n+1}\Lambda\right|^{\frac32} +\frac32 \left|\overline{f}^{n+1}\Lambda\right|^{\frac32}\\
    &+\frac32 \left|\overline{f}^{n+1}\Lambda\right|^{\frac12}L[f^{n+1}]\cdot\nabla_{\bv}\Lambda \left|\overline{f}^{n+1}\right| +\frac32 \left|\overline{f}^{n+1}\Lambda\right|^{\frac12}(\bu^n-\bv)\cdot\nabla_{\bv}\Lambda \left|\overline{f}^{n+1}\right|\\
    &-\frac32 \left|\overline{f}^{n+1}\Lambda\right|^{\frac12}\text{sgn}\overline{f}^{n+1} \bigg(\nabla_{\bv}\cdot L[\overline{f}^{n+1}]f^n +L[\overline{f}^{n+1}]\cdot\nabla_{\bv} f^n \bigg)\Lambda\\ &-\frac32 \left|\overline{f}^{n+1}\Lambda\right|^{\frac12}\text{sgn}\overline{f}^{n+1}\overline{\bu}^n \cdot\nabla_{\bv} f^n \Lambda.
  \end{aligned}
\end{equation}
Integrating \eqref{eq-app-kcswt-dif-frac} over $\bbr^3\times \bbr^3$ gives
\begin{equation}\label{eq-app-kcswt-dif-frac-gron}
  \begin{aligned}
    &\frac{d}{d t}\left|\overline{f}^{n+1}\Lambda\right|_{L^{\frac32}}^{\frac32}\\
    =&\int_{\bbr^3}\int_{\bbr^3}\Bigg(-\frac12 \nabla\cdot L[f^{n+1}]\left|\overline{f}^{n+1}\Lambda\right|^{\frac32} +\frac32 \left|\overline{f}^{n+1}\Lambda\right|^{\frac32}\Bigg)d\bx d\bv \\
    &+\int_{\bbr^3}\int_{\bbr^3}\Bigg(\frac32 \left|\overline{f}^{n+1}\Lambda\right|^{\frac12}L[f^{n+1}]\cdot\nabla_{\bv}\Lambda \left|\overline{f}^{n+1}\right|\\ &\qquad \qquad +\frac32 \left|\overline{f}^{n+1}\Lambda\right|^{\frac12}(\bu^n-\bv)\cdot\nabla_{\bv}\Lambda \left|\overline{f}^{n+1}\right|\Bigg)d\bx d\bv\\
    &-\int_{\bbr^3}\int_{\bbr^3} \frac32 \left|\overline{f}^{n+1}\Lambda\right|^{\frac12}\text{sgn}\overline{f}^{n+1} \bigg(\nabla_{\bv}\cdot L[\overline{f}^{n+1}]f^n +L[\overline{f}^{n+1}]\cdot\nabla_{\bv} f^n \bigg)\Lambda d\bx d\bv\\ &-\int_{\bbr^3}\int_{\bbr^3} \frac32 \left|\overline{f}^{n+1}\Lambda\right|^{\frac12}\text{sgn}\overline{f}^{n+1}\overline{\bu}^n \cdot\nabla_{\bv} f^n \Lambda d\bx d\bv\\
    =:&\sum_{i=1}^4N_i.
  \end{aligned}
\end{equation}
We estimate each $N_i$ $(i=1,2,3,4)$ as follows.
\begin{equation*}
  \begin{aligned}
     |N_1|\le& C|f^{n+1}|_{L^1}\left|\overline{f}^{n+1}\Lambda\right|_{L^{\frac32}}^{\frac32} +C\left|\overline{f}^{n+1}\Lambda\right|_{L^{\frac32}}^{\frac32}\\
          \le& C|f^{n+1}|_{L_{\omega}^2}\left|\overline{f}^{n+1}\Lambda\right|_{L^{\frac32}}^{\frac32} +C\left|\overline{f}^{n+1}\Lambda\right|_{L^{\frac32}}^{\frac32};\\
     |N_2|\le& C|f^{n+1}\langle \bv \rangle|_{L^1}\left|\overline{f}^{n+1}\Lambda\right|_{L^{\frac32}}^{\frac32} +C\Big(1+|\bu^n|_{L^{\infty}}\Big) \left|\overline{f}^{n+1}\Lambda\right|_{L^{\frac32}}^{\frac32}\\
         \le& C|f^{n+1}|_{L_{\omega}^2}\left|\overline{f}^{n+1}\Lambda\right|_{L^{\frac32}}^{\frac32}+C\Big(1+|\bu^n|_{L^{\infty}}\Big) \left|\overline{f}^{n+1}\Lambda\right|_{L^{\frac32}}^{\frac32};\\
     |N_3|\le& C\left|\overline{f}^{n+1}\Lambda\right|_{L^{\frac32}}^{\frac12} \left|\overline{f}^{n+1}\right|_{L^1} |f^n\Lambda|_{L^{\frac32}}\\ &+C\left|\overline{f}^{n+1}\Lambda\right|_{L^{\frac32}}^{\frac12} \left|\overline{f}^{n+1}\langle \bv \rangle\right|_{L^1} |\nabla_{\bv}f^n \langle \bv \rangle \Lambda |_{L^{\frac32}}\\
     \le& C\bigg(|f^n\Lambda|_{L^{\frac32}} +|\nabla_{\bv}f^n \langle \bv \rangle \Lambda|_{L^{\frac32}}\bigg) \left|\overline{f}^{n+1}\Lambda\right|_{L^{\frac32}}^{\frac12} \left|\overline{f}^{n+1}\langle \bv \rangle\right|_{L^1}\\
     \le& C|f^n|_{H_{\omega}^1} \left|\overline{f}^{n+1}\Lambda\right|_{L^{\frac32}}^{\frac12} \left|\overline{f}^{n+1}\langle \bv \rangle\right|_{L^1};\\
     |N_4|\le& C\left|\overline{f}^{n+1}\Lambda\right|_{L^{\frac32}}^{\frac12} \left|\nabla\overline{\bu}^n\right|_{L^2} |\nabla_{\bv}f^n\Lambda|_{L^{2}}\\
     \le& C|\nabla_{\bv}f^n|_{L_{\omega}^2}\left|\overline{f}^{n+1}\Lambda\right|_{L^{\frac32}}^{\frac12} \left|\nabla\overline{\bu}^n\right|_{L^2}.
  \end{aligned}
\end{equation*}
In the above estimates, we have used the following inequalities.
\begin{equation*}
  \begin{aligned}
     &|f^{n+1}\langle \bv \rangle|_{L^1}\\
     \le& \Bigg(\int_{\bbr^3}\int_{\bbr^3}|f^{n+1}|^2(1+\bv^2)(1+\bx^2+\bv^2)^{3\gamma} d\bx d\bv \Bigg)^{\frac12}|(1+\bx^2+\bv^2)^{-\frac{3\gamma}{2}}|_{L^2}\\
     \le& C|f^{n+1}|_{L_{\omega}^2};
  \end{aligned}
\end{equation*}
\begin{equation*}
  \begin{aligned}
     |f^n\Lambda|_{L^{\frac32}}\le& \Bigg(\int_{\bbr^3}\int_{\bbr^3}|f^n|^2(1+\bv^2)^{2\alpha}(1+\bx^2+\bv^2)^{\gamma} d\bx d\bv \Bigg)^{\frac12}|(1+\bx^2+\bv^2)^{-\frac{\gamma}{2}}|_{L^6}\\
     \le& C|f^n|_{L_{\omega}^2};
  \end{aligned}
\end{equation*}
\begin{equation*}
  \begin{aligned}
     &|\nabla_{\bv}f^n \langle \bv \rangle \Lambda|_{L^{\frac32}}\\
     \le& \Bigg(\int_{\bbr^3}\int_{\bbr^3}|\nabla_{\bv}f^n|^2(1+\bv^2)^{1+2\alpha} (1+\bx^2+\bv^2)^{\gamma} d\bx d\bv \Bigg)^{\frac12}|(1+\bx^2+\bv^2)^{-\frac{\gamma}{2}}|_{L^6}\\
     \le& C|\nabla_{\bv}f^n|_{L_{\omega}^2}.
   \end{aligned}
\end{equation*}
Substituting the estimates on $N_i$ $(i=1,2,3,4)$ into \eqref{eq-app-kcswt-dif-frac-gron}, we deduce that
\begin{equation}\label{eq-app-kcswt-dif-frac-gronshort}
  \begin{aligned}
    \frac{d}{d t}\left|\overline{f}^{n+1}\Lambda\right|_{L^{\frac32}}^{2}
    \le&\bigg(C +C|f^{n+1}|_{L_{\omega}^2} +C|\bu^n|_{L^{\infty}} +C|f^n|_{H_{\omega}^1}^2 \bigg)\left|\overline{f}^{n+1}\Lambda\right|_{L^{\frac32}}^{2}\\
    &+\left|\overline{f}^{n+1}\langle \bv \rangle\right|_{L^1}^2 +\frac18 \left|\nabla\overline{\bu}^n\right|_{L^2}^2.
  \end{aligned}
\end{equation}
Similarly, we have
\begin{equation}\label{eq-app-kcswt-dif-lonesqu-gron}
  \begin{aligned}
    \frac{d}{d t}\left|\overline{f}^{n+1}\langle \bv \rangle\right|_{L^1}^2
    \le&\bigg(C +C|f^{n+1}|_{L_{\omega}^2} +C|\bu^n|_{L^{\infty}} +C|f^n|_{H_{\omega}^1}^2 \bigg)\left|\overline{f}^{n+1}\langle \bv \rangle\right|_{L^1}^2 \\ &+\frac18  \left|\nabla\overline{\bu}^n\right|_{L^2}^2.
  \end{aligned}
\end{equation}
Define
\[
 F^{n+1}(t):=\left|\overline{\bu}^{n+1}\right|_{L^2}^2 +\left|\overline{f}^{n+1}\Lambda\right|_{L^{\frac32}}^2 +\left|\overline{f}^{n+1}\langle \bv \rangle\right|_{L^1}^2.
\]
Combining \eqref{eq-app-udif-squashort}, \eqref{eq-app-kcswt-dif-frac-gronshort}, and \eqref{eq-app-kcswt-dif-lonesqu-gron} , we obtain
\begin{equation}\label{eq-app-dif-gron}
  \begin{aligned}
    &\frac{d}{dt}F^{n+1} +\left|\nabla\overline{\bu}^{n+1}\right|_{L^2}^2\\
    \le& \bigg(C +C|f^{n+1}|_{L_{\omega}^2} +C|\bu^n|_{L^{\infty}} +C|f^n|_{H_{\omega}^1}^2 +C\left|\nabla{\bu}^{n+1}\right|_{L^2}^2 \bigg) F^{n+1} +\frac14 \left|\nabla\overline{\bu}^n\right|_{L^2}^2.
  \end{aligned}
\end{equation}
Solving the above Gronwall inequality in $[0, T_0]$ $(0<T_0\le T_*)$, we obtain
\begin{equation}\label{eq-app-dif-est}
  \sup_{0\le t\le T_0}F^{n+1}(t) +\int_0^{T_0}\left|\nabla\overline{\bu}^{n+1}(t)\right|_{L^2}^2 dt
  \le \frac{A(T_0)}{4} \int_0^{T_0}\left|\nabla\overline{\bu}^n(t)\right|_{L^2}^2 dt,
\end{equation}
where $A(T_0)$ is given by
\[
 A(T_0):=\exp\left(\int_0^{T_0} \Big( C +C|f^{n+1}|_{L_{\omega}^2} +C|\bu^n|_{L^{\infty}} +C|f^n|_{H_{\omega}^1}^2 +C\left|\nabla{\bu}^{n+1}\right|_{L^2}^2 \Big) dt \right).
\]
Using the uniform bound on the approximate solutions, we take $T_0$ suitably small, so that
\[
 \exp\left(\int_0^{T_0} \Big( C +C|f^{n+1}|_{L_{\omega}^2} +C|\bu^n|_{L^{\infty}} +C|f^n|_{H_{\omega}^1}^2 +C\left|\nabla{\bu}^{n+1}\right|_{L^2}^2 \Big) dt \right) \le 2.
\]
Thus, we have
\begin{equation}\label{eq-app-dif-est-tzro}
  \sup_{0\le t\le T_0}F^{n+1}(t) +\int_0^{T_0}\left|\nabla\overline{\bu}^{n+1}(t)\right|_{L^2}^2 dt
  \le \frac{1}{2} \int_0^{T_0}\left|\nabla\overline{\bu}^n(t)\right|_{L^2}^2 dt.
\end{equation}
Summing \eqref{eq-app-dif-est-tzro} over all $n\in \bbn$ gives
\begin{equation}\label{eq-app-dif-est-sum}
  \sup_{0\le t\le T_0}\sum_{n=2}^{\infty}F^n(t) +\frac{1}{2}\sum_{n=2}^{\infty} \int_0^{T_0}\left|\nabla\overline{\bu}^n(t)\right|_{L^2}^2 dt
  \le \frac{1}{2} \int_0^{T_0}\left|\nabla\overline{\bu}^1(t)\right|_{L^2}^2 dt.
\end{equation}
We deduce from \eqref{eq-app-dif-est-sum} that there exists $(f, \bu)$ such that
\begin{equation}\label{eq-app-conver-lowoder}
  \begin{aligned}
    &f^n \to f, \quad\text{in $C([0, T_0]; L^1)$, as $n \to \infty$};\\
    &\bu^n \to \bu, \quad \text{in $C(0, T_0; L^2)$, as $n \to \infty$};\\
    &\bu^n \to \bu, \quad \text{in $L^2(0, T_0; D^1)$, as $n \to \infty$}.
  \end{aligned}
\end{equation}
From \eqref{eq-app-conver-lowoder}, it is easy to show that $(f, \bu)$ verifies \eqref{eq-cs-s} in the sense of distributions.
\vskip 3mm
\noindent \textbf{Continuity in Time}
\vskip 3mm
By induction, we know \eqref{eq-appro-f-norm} and \eqref{eq-appro-induassunext} hold for all $n\in \bbn$. Using uniqueness of the weak limit, we deduce by \eqref{eq-app-conver-lowoder} that
\begin{equation}\label{eq-app-conver-wek}
  \begin{aligned}
    f^n &\rightharpoonup f, \quad\text{weakly-$\star$ in $L^{\infty}(0, T_0; H^1_{\omega})$, as $n \to \infty$};\\
    \bu^n &\rightharpoonup \bu, \quad \text{weakly-$\star$ in $L^{\infty}(0, T_0; H^2)$, as $n \to \infty$};\\
    \bu^n_t &\rightharpoonup \bu_t, \quad \text{weakly in $L^2(0, T_0; H^1)$, as $n \to \infty$};\\
    \bu^n &\rightharpoonup \bu, \quad \text{weakly in $L^2(0, T_0; D^{2,q})$, as $n \to \infty$}.
  \end{aligned}
\end{equation}
It follows from \eqref{eq-app-conver-lowoder} and \eqref{eq-app-conver-wek} that
\begin{equation}\label{eq-app-conti-wek}
  \begin{aligned}
    &\bu_t \in L^2(0, T_0; H^1),
    \quad \bu \in L^2(0, T_0; D^{2,q}),\\
    &\bu \in C([0, T_0]; H^1)\cap C([0, T_0]; H^2-W),
  \end{aligned}
\end{equation}
where $C([0, T_0]; H^2-W)$ means continuity in $[0, T_0]$ with respect to the weak topology in $H^2$.
Using the regularity of $\bu$, we can also demonstrate that
\begin{equation}\label{eq-conti-fwtsobo}
 f \in C([0, T_0]; H^1_{\omega})
\end{equation}
by the same proof as in [\cite{jin2019local}, Proposition 2.1]. From $\eqref{eq-cs-s}_2$, we infer that $\bu_{tt} \in L^2(0, T_0; H^{-1})$. This together with $\eqref{eq-app-conti-wek}_1$ gives
\begin{equation}\label{eq-conti-udift}
  \bu_t  \in C([0, T_0]; L^2).
\end{equation}
Project $\eqref{eq-cs-s}_2$ on the divergence-free fields. We obtain
\begin{equation}\label{eq-cs-s-df}
  \Delta \bu =\bu_t-\mathcal P \int_{\bbr^3} f(\bv-\bu) d\bv.
\end{equation}
Using elliptic estimates on \eqref{eq-cs-s-df}, and Lemma \ref{lem-heim-deco}, we deduce that for all $t_1, t_2 \in [0, T_0]$
\begin{equation}\label{eq-conti-usec}
  \begin{aligned}
    &\left|\bu(t_2)-\bu(t_1)\right|_{D^2}\\ \le& C\left|\bu_t(t_2)-\bu_t(t_1)\right|_{L^2}\\ &+C\left|\int_{\bbr^3}f(t_2)\left(\bv-\bu(t_2)\right)d\bv-\int_{\bbr^3}f(t_1) \left(\bv-\bu(t_1)\right)d\bv\right|_{L^2}\\
    \le&  C\left|\bu_t(t_2)-\bu_t(t_1)\right|_{L^2} +C\left|\int_{\bbr^3}\left[f(t_2)-f(t_1)\right] \bv d\bv\right|_{L^2}\\&+ C\left|\int_{\bbr^3}\left[f(t_2)-f(t_1)\right] d\bv \bu(t_1)\right|_{L^2} + C\left|\int_{\bbr^3}f(t_2) d\bv \left[\bu(t_2)-\bu(t_1)\right] \right|_{L^2}\\
    \le& C\left|\bu_t(t_2)-\bu_t(t_1)\right|_{L^2} +C\left( 1+ \left|\bu(t_1)\right|_{\infty}\right)\left|f(t_2)-f(t_1)\right|_{L_{\omega}^2}\\ &+C\left|f(t_2)\right|_{H_{\omega}^1}\left|\bu(t_2)-\bu(t_1)\right|_{D^1}.
  \end{aligned}
\end{equation}
From $\eqref{eq-app-conti-wek}_2$, \eqref{eq-conti-fwtsobo} and \eqref{eq-conti-udift}, we know
\begin{equation}\label{eq-conti-gradutwo}
  \bu \in C([0, T_0]; D^2).
\end{equation}
By virtue of  \eqref{eq-app-conti-wek}, it is easy to find that
\[
 \bu \in C([0, T_0]; H^2)\cap L^2(0, T_0; D^{2,q}).
\]
Therefore, $(f,\bu)$ is the desired strong solution in the sense of Definition \ref{def-stro}. The uniqueness of strong solutions can be proved in the same way as in the derivation of \eqref{eq-app-dif-gron}.
This completes the proof. $\hfill \square$
%
%
\section{A Priori Estimates}\label{sec-apriori}
\setcounter{equation}{0}
In this section, we derive some a priori estimates on the coupled model. Define the energy of the system as
\[
 E(t):=\frac12 \int_{\bbr^3} \bu^2(t,\bx) d\bx +\frac12 \int_{\bbr^6} f(t,\bx,\bv)\bv^2 d\bx d\bv,
\]
and the initial energy $E_0:=E(0)$.
\begin{lemma}\label{lm-cs-s-emapriori}
Under the conditions in Theorem \ref{thm-exist}, if $(f,\bu)$ is a classical solution to \eqref{eq-cs-s}-\eqref{eq-sys-inidata}, then it holds for all $T \in (0, \infty)$ that
\[
 \begin{aligned}
   (i)&\ |f(T)|_{L^1}= |f_0|_{L^1};\\
   (ii)&\ f\ge0 \ \text{and} \ |f(T)|_{L^{\infty}}\le |f_0|_{L^{\infty}}\exp\big(CT\big),\quad \text{where $C:=C(\varphi, f_0)$};\\
   (iii)&\ E(T)+\frac12\int_0^T \int_{\bbr^6}\int_{\bbr^6}\varphi(|\bx-\by|)f(t,\by,\bv^*)f(t,\bx,\bv)(\bv^*-\bv)^2d\by d\bv^* d\bx d\bv dt\\
   & \qquad + \int_0^T \left|\nabla \bu(t) \right|_{L^2}^2 d t +\int_0^T \int_{\bbr^6} f(t,\bx,\bv)(\bu-\bv)^2 d\bx d\bv dt=E_0.
 \end{aligned}
\]
\end{lemma}
\begin{proof}
(i)\ From $f_0 \in H^1_{\omega}$, we deduce that
\begin{equation} \label{eq-ini-f-lonenm}
 \begin{aligned}
 |f_0|_{L^1}&=\int_{\bbr^6}f_0(\bx,\bv)\omega^{\frac12}(\bx,\bv)\omega^{-\frac12}(\bx,\bv)d\bx d\bv\\
   &\le |\omega^{-\frac12}|_{L^2}|f_0|_{L^2_{\omega}}\le C|f_0|_{L^2_{\omega}}.
 \end{aligned}
\end{equation}
Integrating $\eqref{eq-cs-s}_1$ over $[0,T]\times \bbr^3 \times \bbr^3$ gives
\begin{equation} \label{eq-kin-conser-mass}
 |f(T)|_{L^1}=|f_0|_{L^1}.
\end{equation}
(ii)\ Denote by $(X(t;\bx_0,\bv_0),V(t;\bx_0,\bv_0))$ the characteristic issuing from $(\bx_0,\bv_0)$. It verifies
\begin{equation} \label{eq-charac}
\begin{dcases}
  \frac{d X}{d t}=V, \\
  \frac{d V}{d t}=\int_{\bbr^{6}} \varphi(|X-\by|)f(t, \by,\bv^*)(\bv^*-V)d \by d \bv^*+\bu(t,X)-V.
\end{dcases}
\end{equation}
\begin{equation} \label{eq-charac-ini}
 X(0;\bx_0,\bv_0)=\bx_0, \qquad V(0;\bx_0,\bv_0)=\bv_0.
\end{equation}
Recall that
\[
 a(t,\bx)=\int_{\bbr^{6}} \varphi(|\bx-\by|)f(t, \by,\bv^*) d \by d \bv^*,
\]
\[
 \mathbf{b} (t,\bx)=\int_{\bbr^{6}} \varphi(|\bx-\by|)f(t, \by,\bv^*) \bv^* d \by d \bv^*.
\]
Solving the equation \eqref{eq-kine-cs} by the method of characteristics gives
\begin{equation} \label{eq-kin-cs-positive}
 f(t,X(t;\bx_0,\bv_0),V(t;\bx_0,\bv_0))=f_0(\bx_0,\bv_0)\exp \left(3 \int_0^t [1+a(\tau,X(\tau))] d \tau \right)\ge 0.
\end{equation}
From \eqref{eq-kin-conser-mass}, \eqref{eq-kin-cs-positive} and the initial condition $f_0(\bx, \bv)\in L^{\infty}(\bbr^3\times\bbr^3)$, we deduce that
\begin{equation} \label{eq-kin-cs-linfty}
  |f(T)|_{L^{\infty}}\le |f_0|_{L^{\infty}}\exp\big(CT\big),\quad \text{where $C:=C(\varphi, f_0)$}.
\end{equation}
(iii)\ Multiplying $\eqref{eq-cs-s}_1$ by $\frac12 \bv^2$, and integrating the resulting equation over $\bbr^3\times\bbr^3$ lead to
\begin{equation}\label{eq-cs-ener-dt}
  \begin{gathered}
    \frac{d}{dt}\int_{\bbr^6}\frac12 f\bv^2d\bx d\bv +\frac12 \int_{\bbr^6}\int_{\bbr^6}\varphi(|\bx-\by|)f(t,\by,\bv^*)f(t,\bx,\bv)(\bv^*-\bv)^2d\by d\bv^* d\bx d\bv \\
    =\int_{\bbr^6}f\bv\cdot(\bu-\bv)d\bx d\bv.
  \end{gathered}
\end{equation}
Multiplying $\eqref{eq-cs-s}_2$ by $\bu$, and integrating the resulting equation over $\bbr^3$ give
\begin{equation}\label{eq-s-ener-dt}
    \frac12 \frac{d}{dt}\int_{\bbr^3} \bu^2 d\bx +|\nabla \bu|_{L^2}^2
    =\int_{\bbr^6}f\bu\cdot(\bv-\bu)d\bx d\bv.
\end{equation}
Adding \eqref{eq-cs-ener-dt} to \eqref{eq-s-ener-dt}, and integrating the resulting equation over $[0,T]$, result in Lemma \ref{lm-cs-s-emapriori}(iii). This completes the proof.
\end{proof}
In order to derive the key estimate on $\int_0^T |\bu(t)|_{L^{\infty}}dt$, we need the following lemma.
\begin{lemma}\label{lm-cs-s-fvwtapriori}
Under the conditions in Theorem \ref{thm-exist}, if $f(t, \bx, \bv)$ is a classical solution to \eqref{eq-cs-s}-\eqref{eq-sys-inidata}, then it holds for all $T \in (0, \infty)$ that
\[
 \left|f(T) \langle \bv \rangle^3\right|_{L^2}\le C\Big(1+T^{\frac32}\Big) \exp \big(CT\big), \quad \text{where $C:=C(\varphi, f_0, E_0)$}.
\]
\end{lemma}
\begin{proof}
Multiplying $\eqref{eq-cs-s}_1$ by $2f\langle \bv \rangle^6$, we obtain
\begin{equation}\label{eq-zerowt-dif}
  \begin{gathered}
   \frac{\partial}{\partial t}(f^2\langle \bv \rangle^6)+\bv \cdot \nabla_{\bx}(f^2\langle \bv \rangle^6)+\nabla_{\bv}\cdot \Big(L[f]f^2\langle \bv \rangle^6+(\bu-\bv)f^2\langle \bv \rangle^6 \Big) \\=L[f]\cdot \nabla_{\bv}\langle \bv \rangle^6 f^2 -\nabla_{\bv}\cdot L[f] f^2\langle \bv \rangle^6 +(\bu-\bv)\cdot \nabla_{\bv}\langle \bv \rangle^6 f^2 +3 f^2\langle \bv \rangle^6.
  \end{gathered}
\end{equation}
Integrating \eqref{eq-zerowt-dif} over $\bbr^3 \times \bbr^3$ yields
\begin{equation}\label{eq-fvwt-gron}
  \begin{aligned}
    \frac{d}{dt}\left|f\langle \bv \rangle^3\right|_{L^2}^2 =&\int_{\bbr^6} \Big[ -\nabla_{\bv}\cdot L[f] |f\langle \bv \rangle^3|^2 +3|f\langle \bv \rangle^3|^2 \Big] d\bx d\bv \\
    &+\int_{\bbr^6} L[f]\cdot \nabla_{\bv}\langle \bv \rangle^6 f^2  d\bx d\bv \\
    &+\int_{\bbr^6} (\bu-\bv)\cdot \nabla_{\bv}\langle \bv \rangle^6  f^2  d\bx d\bv \\
    =:&I_1+I_2+I_3.
  \end{aligned}
\end{equation}
We estimate each $I_i$ $(i=1, 2, 3)$ as follows.
\begin{equation*}
  \begin{aligned}
    I_1=&\int_{\bbr^6} \Big[ -\nabla_{\bv}\cdot L[f] |f\langle \bv \rangle^3|^2 +3|f\langle \bv \rangle^3|^2 \Big] d\bx d\bv \\
      \le& C(\varphi, f_0) \left|f\langle \bv \rangle^3\right|_{L^2}^2;\\
    I_2=&\int_{\bbr^6} L[f]\cdot \nabla_{\bv}\langle \bv \rangle^6 f^2  d\bx d\bv  \\
      \le& C(\varphi) \left|f \langle \bv \rangle \right|_{L^1} \left|f\langle \bv \rangle^3\right|_{L^2}^2\\
      \le& C(\varphi) \left|f \right|_{L^1}^{\frac12}\left|f \langle \bv \rangle^2 \right|_{L^1}^{\frac12} \left|f\langle \bv \rangle^3\right|_{L^2}^2\\
      \le& C(\varphi, f_0, E_0) \left|f\langle \bv \rangle^3\right|_{L^2}^2;\\
    I_3=&\int_{\bbr^6} (\bu-\bv)\cdot \nabla_{\bv}\langle \bv \rangle^6  f^2  d\bx d\bv \\
       \le& C\left|\nabla \bu \right|_{L^2} \left|f\langle \bv \rangle^2\right|_{L^3} \left|f\langle \bv \rangle^3\right|_{L^2} +C \left|f\langle \bv \rangle^3\right|_{L^2}^2\\
       \le& C\left|\nabla \bu \right|_{L^2} \left|f\right|_{L^{\infty}}^{\frac13} \left|f\langle \bv \rangle^3\right|_{L^2}^{\frac53} +C \left|f\langle \bv \rangle^3\right|_{L^2}^2\\
       \le& C(f_0) \exp\big(C(\varphi, f_0)t \big)\left|\nabla \bu \right|_{L^2} \left|f\langle \bv \rangle^3\right|_{L^2}^{\frac53} +C \left|f\langle \bv \rangle^3\right|_{L^2}^2,\\
  \end{aligned}
\end{equation*}
where in the above estimate, we have used Lemma \ref{lm-cs-s-emapriori}. Substituting the estimates on $I_i$ $(i=1, 2, 3)$ into \eqref{eq-fvwt-gron}, we get
\begin{equation}\label{eq-fvwt-gron-short}
    \frac{d}{dt}\left|f\langle \bv \rangle^3\right|_{L^2}^2 \le C(\varphi, f_0, E_0) \left|f\langle \bv \rangle^3\right|_{L^2}^2
    +C(f_0) \exp\big(C(\varphi, f_0)t \big)\left|\nabla \bu \right|_{L^2} \left|f\langle \bv \rangle^3\right|_{L^2}^{\frac53}.
\end{equation}
Solving the above Gronwall's inequality yields
\[
 \left|f(T) \langle \bv \rangle^3\right|_{L^2}\le C\Big(1+T^{\frac32}\Big) \exp \big(CT\big), \quad \text{where $C:=C(\varphi, f_0, E_0)$}.
\]
This completes the proof.
\end{proof}
We use Lemma \ref{lm-cs-s-fvwtapriori} to estimate $\int_0^T\left|\bu(t)\right|_{L^{\infty}} dt$ for all $T>0$. Then the estimate on $\int_0^T\left|\nabla \bu(t)\right|_{L^{\infty}} dt$ can be obtained by a bootstrap argument.
\begin{lemma}\label{lm-cs-s-graduwoneinftapriori}
Under the conditions in Theorem \ref{thm-exist}, if $(f, \bu)$ is a classical solution to \eqref{eq-cs-s}-\eqref{eq-sys-inidata}, then it holds for all $T \in (0, \infty)$ that
\[
 \int_0^T\left|\bu\right|_{L^{\infty}} dt\le C\Big(1+T^{\frac72}\Big) \exp \big(CT\big), \quad \text{where $C:=C(\varphi, f_0, \bu_0, E_0)$}.
\]
and
\[
 \int_0^T\left|\nabla \bu \right|_{L^{\infty}} dt\le C\big(1+T\big) \exp \Big(C(1+ T^{\frac72})e^{CT}\Big), \quad \text{where $C:=C(q, \varphi, R_0, f_0, \bu_0, E_0)$}.
\]
\end{lemma}
\begin{proof}
Using Proposition \ref{prop-s}, we know
\begin{equation}\label{eq-stwo-apriori}
  \int_0^T \left|\bu_t \right|_{L^2}^2 dt+ \int_0^T \left|\nabla^2\bu \right|_{L^2}^2 dt \le C\left(\left|\bu_0\right|_{H^2}^2+\int_0^T \left|\int_{\bbr^3} f (\bv-\bu )d\bv \right|_{L^2}^2 dt\right).
\end{equation}
It is easy to see that
\begin{equation}\label{eq-couptermltwo}
  \begin{aligned}
    \left|\int_{\bbr^3} f (\bv-\bu )d\bv \right|_{L^2}\le& \left|\int_{\bbr^3} f \bv d\bv\right|_{L^2} +\left|\int_{\bbr^3} f  d\bv\right|_{L^2} \left|\bu \right|_{L^{\infty}}\\
    \le& C\left|f\langle \bv \rangle^3\right|_{L^2}+ C\left|f\langle \bv \rangle^3\right|_{L^2} \left|\nabla \bu \right|_{L^2}^{\frac12} \left|\nabla^2 \bu \right|_{L^2}^{\frac12},
  \end{aligned}
\end{equation}
where we have used the inequality
\[
 \left|\int_{\bbr^3}f \langle \bv \rangle d\bv \right|_{L^2} \le \left|f\langle \bv \rangle^3\right|_{L^2}\left|\langle \bv \rangle^{-2}\right|_{L^2}\le C\left|f\langle \bv \rangle^3\right|_{L^2}
\]
and the Sobolev inequality
\begin{equation}\label{eq-sobo-linfty}
 \left|\bu \right|_{L^{\infty}}  \le C \left|\nabla \bu \right|_{L^2}^{\frac12} \left|\nabla^2 \bu \right|_{L^2}^{\frac12} \quad \text{in $\bbr^3$.}
\end{equation}
From \eqref{eq-couptermltwo}, we deduce that
\begin{equation}\label{eq-couptermltwo-est}
  \begin{aligned}
    &C\int_0^T \left|\int_{\bbr^3} f (\bv-\bu )d\bv \right|_{L^2}^2 dt\\
    \le& \int_0^T \left[C\left|f\langle \bv \rangle^3\right|_{L^2}^2+ C\left|f\langle \bv \rangle^3\right|_{L^2}^4 \left|\nabla \bu \right|_{L^2}^2 +\frac12 \left|\nabla^2 \bu \right|_{L^2}^2 \right] dt\\
    \le& CT \sup_{0 \le t \le T} \left|f\langle \bv \rangle^3\right|_{L^2}^2 +C \sup_{0 \le t \le T} \left|f\langle \bv \rangle^3\right|_{L^2}^4 \int_0^T \left|\nabla \bu \right|_{L^2}^2 dt +\frac12 \int_0^T \left|\nabla^2 \bu \right|_{L^2}^2 dt.
  \end{aligned}
\end{equation}
Substituting \eqref{eq-couptermltwo-est} into \eqref{eq-stwo-apriori}, we obtain by Lemma \ref{lm-cs-s-emapriori} and \ref{lm-cs-s-fvwtapriori} that
\begin{equation}\label{eq-gradtwou-ltwo}
  \begin{aligned}
    \int_0^T \left|\nabla^2\bu \right|_{L^2}^2 dt \le& C\left|\bu_0\right|_{H^2}^2+ CT \sup_{0 \le t \le T} \left|f\langle \bv \rangle^3\right|_{L^2}^2 +C \sup_{0 \le t \le T} \left|f\langle \bv \rangle^3\right|_{L^2}^4 \int_0^T \left|\nabla \bu \right|_{L^2}^2 dt\\
    \le& C (1+T^6) \exp\Big(CT \Big), \quad \text{where $C:=C(\varphi, f_0, \bu_0, E_0)$}.
  \end{aligned}
\end{equation}
Using \eqref{eq-sobo-linfty} again, we deduce that
\begin{equation}\label{eq-ulinfty-est}
  \begin{aligned}
    \int_0^T\left|\bu\right|_{L^{\infty}} dt\le& C\int_0^T \left|\nabla \bu \right|_{L^2}^{\frac12} \left|\nabla^2 \bu \right|_{L^2}^{\frac12} dt\\
      \le& C \int_0^T \left(\left|\nabla \bu \right|_{L^2} + \left|\nabla^2 \bu \right|_{L^2}\right) dt\\
      \le& C T^{\frac12} \left[ \left(\int_0^T\left|\nabla\bu\right|_{L^2}^2 dt \right)^{\frac12} + \left(\int_0^T\left|\nabla^2\bu\right|_{L^2}^2 dt \right)^{\frac12} \right]\\
      \le& C\Big(1+T^{\frac72}\Big) \exp \big(CT\big), \quad \text{where $C:=C(\varphi, f_0, \bu_0, E_0)$}.
  \end{aligned}
\end{equation}
Multiplying $\eqref{eq-cs-s}_1$ by $\langle \bv \rangle^k $ yields
\begin{equation}\label{eq-cs-fvwt}
  \begin{aligned}
    &\frac{\partial}{\partial t} \Big(f \langle \bv \rangle^k \Big)+ \bv \cdot \nabla_{\bx}\Big(f \langle \bv \rangle^k \Big)+\nabla_{\bv} \cdot \Big(L[f] f \langle \bv \rangle^k+(\bu-\bv) f \langle \bv \rangle^k\Big)\\
    =& fL[f]\cdot \nabla_{\bv} \langle \bv \rangle^k+ f (\bu-\bv) \cdot \nabla_{\bv} \langle \bv \rangle^k.
  \end{aligned}
\end{equation}
Multiplying \eqref{eq-cs-fvwt} by $q \Big(f \langle \bv \rangle^k \Big)^{q-1} $, we obtain
\begin{equation}\label{eq-cs-fvwtq}
  \begin{aligned}
    &\frac{\partial}{\partial t} \Big(f \langle \bv \rangle^k \Big)^q+ \bv \cdot \nabla_{\bx}\Big(f \langle \bv \rangle^k \Big)^q+\nabla_{\bv} \cdot \Big(L[f] \Big(f \langle \bv \rangle^k \Big)^{q}+(\bu-\bv) \Big(f \langle \bv \rangle^k \Big)^{q}\Big)\\
    =&-(q-1) \nabla_{\bv} \cdot L[f]\Big(f \langle \bv \rangle^k \Big)^q+ 3(q-1)\Big(f \langle \bv \rangle^k \Big)^q\\
    &+q fL[f]\cdot \nabla_{\bv} \langle \bv \rangle^k \Big(f \langle \bv \rangle^k \Big)^{q-1} + q f (\bu-\bv) \cdot \nabla_{\bv} \langle \bv \rangle^k \Big(f \langle \bv \rangle^k \Big)^{q-1}.
  \end{aligned}
\end{equation}
Integrating \eqref{eq-cs-fvwtq} over $\bbr^3 \times \bbr^3$ and using Lemma \ref{lm-cs-s-emapriori}, we have
\begin{equation}\label{eq-fvwtlq-gron}
  \begin{aligned}
    \frac{d}{dt}\left|f\langle \bv \rangle^k\right|_{L^q}^q \le& C(q, \varphi) \left|f\langle \bv \rangle\right|_{L^1}\left|f\langle \bv \rangle^k\right|_{L^q}^q +C(q)\left(1+ \left|\bu \right|_{L^\infty}\right)\left|f\langle \bv \rangle^k\right|_{L^q}^q\\
    \le& C(q, \varphi, f_0, E_0) \left|f\langle \bv \rangle^k\right|_{L^q}^q +C(q) \left|\bu \right|_{L^\infty} \left|f\langle \bv \rangle^k\right|_{L^q}^q.
  \end{aligned}
\end{equation}
Using the assumption that $\text{supp}_{\bv}f_0(\bx,\cdot)\subseteq B(R_0)$ for all $\bx \in \bbr^3$ and \eqref{eq-ulinfty-est}, we solve the above Gronwall's inequality to obtain
\begin{equation}\label{eq-fvwt-lqest}
  \begin{aligned}
    \left|f(t)\langle \bv \rangle^k\right|_{L^q} \le& \left|f_0 \langle \bv \rangle^k\right|_{L^q} \exp\left(Ct+ C\int_0^t \left|\bu \right|_{L^\infty} d\tau \right)\\
    \le& C \exp\left(C\Big(1+ t^{\frac72}\Big) e^{Ct} \right),\quad \text{where $C:=C(q, \varphi, R_0, f_0, \bu_0, E_0)$}.
  \end{aligned}
\end{equation}
Employing Proposition \ref{prop-s} again, we know
\begin{equation}\label{eq-sq-apriori}
  \int_0^T \left|\bu_t \right|_{L^q}^2 dt+ \int_0^T \left|\nabla^2\bu \right|_{L^q}^2 dt \le C\left(\left|\bu_0\right|_{H^2}^2+\int_0^T \left|\int_{\bbr^3} f (\bv-\bu )d\bv \right|_{L^q}^2 dt\right).
\end{equation}
It is easy to see that
\begin{equation}\label{eq-couptermlq}
  \begin{aligned}
    \left|\int_{\bbr^3} f (\bv-\bu )d\bv \right|_{L^q}\le& \left|\int_{\bbr^3} f \bv d\bv\right|_{L^q} +\left|\int_{\bbr^3} f  d\bv\right|_{L^q} \left|\bu \right|_{L^{\infty}}\\
    \le& C\left|f\langle \bv \rangle^4\right|_{L^q}+ C\left|f\langle \bv \rangle^4\right|_{L^q} \left|\nabla \bu \right|_{L^2}^{1-\theta_1} \left|\nabla^2 \bu \right|_{L^q}^{\theta_1},
  \end{aligned}
\end{equation}
where we have used the inequality
\[
 \left|\int_{\bbr^3}f \langle \bv \rangle d\bv \right|_{L^q} \le \left|f\langle \bv \rangle^4\right|_{L^q}\left|\langle \bv \rangle^{-3}\right|_{L^{\frac{q}{q-1}}}\le C\left|f\langle \bv \rangle^4\right|_{L^q}
\]
and the Sobolev inequality in $\bbr^3$,
\begin{equation}\label{eq-soboq-linfty}
 \left|\bu \right|_{L^{\infty}}  \le C \left|\nabla \bu \right|_{L^2}^{1-\theta_1} \left|\nabla^2 \bu \right|_{L^q}^{\theta_1} \quad \text{with $-\frac{1-\theta_1}{2}+ \left(2- \frac3q \right) \theta_1 =0$.}
\end{equation}
From \eqref{eq-couptermlq}, we deduce that
\begin{equation}\label{eq-couptermlq-est}
  \begin{aligned}
    &C\int_0^T \left|\int_{\bbr^3} f (\bv-\bu )d\bv \right|_{L^q}^2 dt\\
    \le& \int_0^T \left[C\left|f\langle \bv \rangle^4\right|_{L^q}^2+ C(q)\left|f\langle \bv \rangle^4\right|_{L^q}^{\frac{2}{1-\theta_1}} \left|\nabla \bu \right|_{L^2}^2 +\frac12 \left|\nabla^2 \bu \right|_{L^q}^2 \right] dt\\
    \le& CT \sup_{0 \le t \le T} \left|f\langle \bv \rangle^4\right|_{L^q}^2 +C(q) \sup_{0 \le t \le T} \left|f\langle \bv \rangle^4\right|_{L^q}^{\frac{2}{1-\theta_1}} \int_0^T \left|\nabla \bu \right|_{L^2}^2 dt +\frac12 \int_0^T \left|\nabla^2 \bu \right|_{L^q}^2 dt.
  \end{aligned}
\end{equation}
Substituting \eqref{eq-couptermlq-est} into \eqref{eq-sq-apriori}, we obtain by Lemma \ref{lm-cs-s-emapriori}, and \eqref{eq-fvwt-lqest} for $k=4$ that
\begin{equation}\label{eq-gradtwou-lq}
  \begin{aligned}
    \int_0^T \left|\nabla^2\bu \right|_{L^q}^2 dt \le& C\left|\bu_0\right|_{H^2}^2+ CT \sup_{0 \le t \le T} \left|f\langle \bv \rangle^4\right|_{L^q}^2 +C \sup_{0 \le t \le T} \left|f\langle \bv \rangle^4\right|_{L^q}^{\frac{2}{1-\theta_1}} \int_0^T \left|\nabla \bu \right|_{L^2}^2 dt\\
    \le& C(1+T) \exp\left(C\Big(1+T^{\frac72}\Big)e^{CT} \right),
  \end{aligned}
\end{equation}
where $C:=C(q, \varphi, R_0, f_0, \bu_0, E_0)$.
Using the Sobolev inequality in $\bbr^3$,
\begin{equation}\label{eq-sobo-gradlinfty}
 \left|\nabla \bu \right|_{L^{\infty}}  \le C \left|\nabla \bu \right|_{L^2}^{1-\theta_2} \left|\nabla^2 \bu \right|_{L^q}^{\theta_2} \quad \text{ with $-\frac{1-\theta_2}{2}+ \left(2- \frac3q \right) \theta_2 =1$,}
\end{equation}
we deduce that
\begin{equation}\label{eq-gradulinfty-est}
  \begin{aligned}
    \int_0^T\left|\nabla \bu\right|_{L^{\infty}} dt\le& C\int_0^T \left|\nabla \bu \right|_{L^2}^{1-\theta_2} \left|\nabla^2 \bu \right|_{L^q}^{\theta_2} dt\\
      \le& C \int_0^T \left(\left|\nabla \bu \right|_{L^2} + \left|\nabla^2 \bu \right|_{L^q}\right) dt\\
      \le& C T^{\frac12} \left[ \left(\int_0^T\left|\nabla\bu\right|_{L^2}^2 dt \right)^{\frac12} + \left(\int_0^T\left|\nabla^2\bu\right|_{L^q}^2 dt \right)^{\frac12} \right]\\
      \le& C(1+T) \exp\left(C\Big(1+T^{\frac72}\Big)e^{CT} \right),
  \end{aligned}
\end{equation}
where $C:=C(q, \varphi, R_0, f_0, \bu_0, E_0)$. Adding \eqref{eq-ulinfty-est} to \eqref{eq-gradulinfty-est} yields
\begin{equation}\label{eq-woneulinfty-est}
  \begin{aligned}
    \int_0^T\left|\bu\right|_{W^{1, \infty}} dt\le& \int_0^T\left|\bu\right|_{L^{\infty}} dt +\int_0^T\left|\nabla \bu\right|_{L^{\infty}} dt\\
      \le& C(1+T) \exp\left(C\Big(1+T^{\frac72}\Big)e^{CT} \right),
  \end{aligned}
\end{equation}
where $C:=C(q, \varphi, R_0, f_0, \bu_0, E_0)$.
This completes the proof.
\end{proof}
With Lemma \ref{lm-cs-s-graduwoneinftapriori} at hand, we then deduce the a priori estimates on $f$ and $\bu$ in the strong solution space.
\begin{lemma}\label{lm-cs-s-apriori}
Under the conditions in Theorem \ref{thm-exist}, if $(f, \bu)$ is a classical solution to \eqref{eq-cs-s}-\eqref{eq-sys-inidata}, then it holds for all $T \in (0, \infty)$ that
\[
 \begin{aligned}
   (i)&\ R(T)\le R_0 + C\Big(1+T^{\frac72}\Big)e^{CT}, \quad \text{where $C:=C(\varphi, f_0, \bu_0, E_0)$};\\
   (ii)&\ \sup_{0 \le t \le T}\left|f(t) \right|_{H_{\omega}^1}\le \left|f_0 \right|_{H_{\omega}^1} \exp\left( C(1+T) \exp\Big(C\big(1+T^{\frac72}\big)e^{CT} \Big)\right),\\ &\quad \text{where $C:=C(q, \varphi, R_0, f_0, \bu_0, E_0)$};\\
   (iii)&\sup_{0 \le t \le T} \left|\bu(t)\right|_{H^2}
      \le C\left(1+\left|f_0 \right|_{H_{\omega}^1}^2\right) \left(1+T^{\frac12}\right) \exp\left( C(1+T) \exp\Big(C\big(1+T^{\frac72}\big)e^{CT} \Big)\right),\\
       &\quad \text{where $C:=C(q, \varphi, R_0, f_0, \bu_0, E_0)$.}
 \end{aligned}
\]
\end{lemma}
\begin{proof}
(i)\ Using Proposition \ref{prop-kine-cs-wp}, Lemma \ref{lm-cs-s-emapriori} and \eqref{eq-woneulinfty-est}, we know
\begin{equation}\label{eq-fvsupp-apri}
  \begin{aligned}
    R(T)\le& R_0+\int_0^T(|\mathbf{b}(t)|_{L^{\infty}}+|\bu(t)|_{L^{\infty}})dt\\
        \le& R_0 + C\Big(1+T^{\frac72}\Big)e^{CT} \quad \text{where $C:=C(\varphi, f_0, \bu_0, E_0)$},
  \end{aligned}
\end{equation}
and
\begin{equation}\label{eq-fwt-apri}
  \begin{aligned}
    \sup_{0 \le t \le T} \left|f(t)\right|_{H_{\omega}^1}\le& \left|f_0\right|_{H_{\omega}^1}\exp \left( C \int_0^T \Big(1+R(t)+|\bu(t)|_{W^{1,\infty}}\Big)d t \right)\\
      \le& \left|f_0 \right|_{H_{\omega}^1} \exp\left( C(1+T) \exp\Big(C\big(1+T^{\frac72}\big)e^{CT} \Big)\right),
  \end{aligned}
\end{equation}
where $C:=C(q, \varphi, R_0, f_0, \bu_0, E_0)$.
Multiplying $\eqref{eq-cs-s}_2$ by $\bu_t$, and integrating the resulting equation over $\bbr^3$ give
\begin{equation}\label{eq-s-gradu-gron}
  \begin{aligned}
    \frac12 \frac{d}{dt}\left|\nabla \bu \right|_{L^2}^2  +\left|\bu_t \right|_{L^2}^2  =& \int_{\bbr^3}\int_{\bbr^3}f (\bv-\bu) \cdot \bu_t d\bv d\bx \\
    \le & C\left|\int_{\bbr^3}f \bv d\bv \right|_{L^{\frac65}}\left|\nabla \bu_t \right|_{L^2} +C\left|\int_{\bbr^3}f d\bv \right|_{L^{\frac32}}\left|\nabla \bu \right|_{L^2} \left|\nabla \bu_t \right|_{L^2}\\
    \le& \frac14 \left|\nabla \bu_t \right|_{L^2}^2 +C\left|f(t)\right|_{L_{\omega}^2}^2 \left(1+\left|\nabla \bu \right|_{L^2}^2 \right).
  \end{aligned}
\end{equation}
Differentiating $\eqref{eq-cs-s}_2$ with respect to $t$ yields
\begin{equation}\label{eq-s-dt}
  \bu_{tt}+\nabla P_t=\Delta \bu_t +\int_{\bbr^3}f_t(\bv-\bu)d\bv -\int_{\bbr^3}f d\bv \bu_t.
\end{equation}
Multiplying \eqref{eq-s-dt} by $\bu_t$, and integrating the resulting equation over $\bbr^3$, we deduce that
\begin{equation}\label{eq-s-udt-gron}
  \begin{aligned}
    &\frac12 \frac{d}{dt}\left| \bu_t \right|_{L^2}^2  +\left|\nabla \bu_t \right|_{L^2}^2\\
      \le& \int_{\bbr^3}\int_{\bbr^3}f_t (\bv-\bu) \cdot \bu_t d\bv d\bx \\
     =& \int_{\bbr^3}\int_{\bbr^3}\left[-\bv \cdot \nabla_{\bx} f- \nabla_{\bv} \cdot (L[f]f+(\bu-\bv)f)\right] (\bv-\bu) \cdot \bu_t d\bv d\bx\\
     =& \int_{\bbr^3}\int_{\bbr^3}f\bv \otimes (\bv-\bu) d\bv : \nabla \bu_t d\bx - \int_{\bbr^3}\int_{\bbr^3}f\bv d\bv \cdot \nabla \bu \cdot \bu_t d\bx  \\
      & +\int_{\bbr^3}\int_{\bbr^3} \Big(L[f]f+(\bu-\bv)f \Big) \cdot \bu_t d\bv d\bx\\
     =:&Q_1+ Q_2.
  \end{aligned}
\end{equation}
We estimate each $Q_i$ ($i=1,2$) as follows.
\begin{equation}\label{eq-s-udt-qone}
 \begin{aligned}
    Q_1=&\int_{\bbr^3}\int_{\bbr^3}f\bv \otimes (\bv-\bu) d\bv : \nabla \bu_t d\bx - \int_{\bbr^3}\int_{\bbr^3}f\bv d\bv \cdot \nabla \bu \cdot \bu_t d\bx \\
     \le &C\left| \int_{\bbr^3}f\bv^2 d\bv \right|_{L^2} \left|\nabla \bu_t \right|_{L^2} +C\left| \int_{\bbr^3}f\bv d\bv \right|_{L^3} \left|\nabla \bu \right|_{L^2} \left|\nabla \bu_t \right|_{L^2}\\
     \le& \frac18 \left|\nabla \bu_t \right|_{L^2}^2 +C\left|f(t)\right|_{H_{\omega}^1}^2 \left(1+\left|\nabla \bu \right|_{L^2}^2 \right);
    \end{aligned}
\end{equation}
\begin{equation}\label{eq-s-udt-qtwo}
  \begin{aligned}
    Q_2=&\int_{\bbr^3}\int_{\bbr^3} \Big(L[f]f+(\bu-\bv)f \Big) \cdot \bu_t d\bv d\bx\\
    \le& C(\varphi) \left| f \langle\bv \rangle \right|_{L^1} \left| f \langle\bv \rangle \right|_{L^{\frac65}} \left|\nabla \bu_t \right|_{L^2}\\ &+C\left| \int_{\bbr^3}f d\bv \right|_{L^{\frac32}} \left|\nabla \bu \right|_{L^2} \left|\nabla \bu_t \right|_{L^2} +C\left| \int_{\bbr^3}f\bv d\bv \right|_{L^{\frac65}} \left|\nabla \bu_t \right|_{L^2}\\
    \le& \frac18 \left|\nabla \bu_t \right|_{L^2}^2 +C(\varphi, f_0, E_0)\left|f(t)\right|_{L_{\omega}^2}^2 \left(1+\left|\nabla \bu \right|_{L^2}^2 \right).
  \end{aligned}
\end{equation}
Substituting \eqref{eq-s-udt-qone} and \eqref{eq-s-udt-qtwo} into \eqref{eq-s-udt-gron} results in
\begin{equation}\label{eq-sudt-gron-short}
  \frac12 \frac{d}{dt}\left| \bu_t \right|_{L^2}^2  +\left|\nabla \bu_t \right|_{L^2}^2  \le \frac14 \left|\nabla \bu_t \right|_{L^2}^2 +C(\varphi, f_0, E_0)\left|f(t)\right|_{H_{\omega}^1}^2 \left(1+\left|\nabla \bu \right|_{L^2}^2 \right).
\end{equation}
Adding \eqref{eq-s-gradu-gron} to \eqref{eq-sudt-gron-short} gives
\begin{equation}\label{eq-sudt-gradu-gron-short}
   \frac{d}{dt}\left( \left|\nabla \bu \right|_{L^2}^2+ \left| \bu_t \right|_{L^2}^2 \right)  +\left|\bu_t \right|_{H^1}^2  \le  C(\varphi, f_0, E_0)\left|f(t)\right|_{H_{\omega}^1}^2 \left(1+\left|\nabla \bu \right|_{L^2}^2 \right).
\end{equation}
Since $\bu_t \in C([0,T];L^2)$, we have
\begin{equation}\label{eq-udt-initi}
  \begin{aligned}
    \left|\bu_t(0) \right|_{L^2}^2=&\left|\Delta \bu_0 +\mathcal P \int_{\bbr^3} f_0(\bv-\bu_0)d\bv \right|_{L^2}^2\\
     \le& C \left|\bu_0 \right|_{D^2}^2 +C\left|f_0 \right|_{L_{\omega}^2}^2 +C\left|f_0 \right|_{L_{\omega}^2}^2\left|\bu_0 \right|_{H^2}^2\\
     \le& C\Big(1+\left|\bu_0 \right|_{H^2}^2\Big) \Big(1+\left|f_0 \right|_{H_{\omega}^1}^2\Big).
  \end{aligned}
\end{equation}
Integrating \eqref{eq-sudt-gradu-gron-short} over $[0, T]$, we obtain by Lemma \ref{lm-cs-s-emapriori}, \eqref{eq-fwt-apri} and \eqref{eq-udt-initi} that
\begin{equation}\label{eq-udt-gradu-apri}
  \begin{aligned}
    &\sup_{0 \le t \le T} \left(\left|\nabla \bu\right|_{L^2}^2 + \left| \bu_t \right|_{L^2}^2\right) +\int_0^T \left|\bu_t \right|_{H^1}^2dt \\
     \le& C(f_0, \bu_0) +C\sup_{0 \le t \le T} \left|f(t) \right|_{H_{\omega}^1}^2 \left( T + \int_0^T \left|\nabla \bu \right|_{L^2}^2dt\right)\\
      \le& C\left(1+\left|f_0 \right|_{H_{\omega}^1}^2\right) (1+T) \exp\left( C(1+T) \exp\Big(C\big(1+T^{\frac72}\big)e^{CT} \Big)\right),
  \end{aligned}
\end{equation}
where $C:=C(q, \varphi, R_0, f_0, \bu_0, E_0)$. Using \eqref{eq-udt-gradu-apri} and the elliptic estimate on \eqref{eq-cs-s-df}, we deduce that
\begin{equation}\label{eq-gradutwo-apri}
  \begin{aligned}
    \sup_{0 \le t \le T} \left|\nabla \bu\right|_{D^2}^2 \le& C \sup_{0 \le t \le T} \left(\left| \bu_t \right|_{L^2}^2 +\left|\int_{\bbr^3} f(\bv-\bu)d\bv \right|_{L^2}^2 \right)\\
     \le&     C \sup_{0 \le t \le T} \left(\left| \bu_t \right|_{L^2}^2 + \left|\int_{\bbr^3} f \bv d\bv \right|_{L^2}^2 + \left|\int_{\bbr^3} f d\bv \right|_{L^3}^2 \left|\nabla \bu\right|_{L^2}^2 \right)\\
      \le& C\left(1+\left|f_0 \right|_{H_{\omega}^1}^4\right)(1+T) \exp\left( C(1+T) \exp\Big(C\big(1+T^{\frac72}\big)e^{CT} \Big)\right),
  \end{aligned}
\end{equation}
where $C:=C(q, \varphi, R_0, f_0, \bu_0, E_0)$. Combining Lemma \ref{lm-cs-s-emapriori}, \eqref{eq-udt-gradu-apri} and \eqref{eq-gradutwo-apri}, we can easily obtain
\begin{equation}\label{eq-uhtwo-apri}
    \sup_{0 \le t \le T} \left|\bu(t)\right|_{H^2}
      \le C\left(1+\left|f_0 \right|_{H_{\omega}^1}^2\right) \left(1+T^{\frac12}\right) \exp\left( C(1+T) \exp\Big(C\big(1+T^{\frac72}\big)e^{CT} \Big)\right).
\end{equation}
where $C:=C(q, \varphi, R_0, f_0, \bu_0, E_0)$.
This completes the proof.
\end{proof}
%
%
\section{Global Existence of Strong Solutions}\label{sec-glob-exst}
\setcounter{equation}{0}
Combining the local-in-time existence result  with the a priori estimates on the coupled model, we present the proof of Theorem \ref{thm-exist}.
\vskip 3mm
\noindent \textit{Proof of Theorem \ref{thm-exist}.} From Proposition \ref{prop-loc-exist}, we know there exists some $T_0>0$ such that \eqref{eq-cs-s}-\eqref{eq-sys-inidata} admits a unique strong solution in $[0, T_0]$. Take the supremum among all the $T_0$, and denote the life span by $T^*$. Next, we prove $T^*=\infty$ by contradiction. Suppose not, i.e., $T^*<\infty$. We mollify the initial data by convolving with the standard mollifier, and then take limit to the approximate classical solutions. It follows from Lemma \ref{lm-cs-s-apriori} that the local strong solutions satisfy
\begin{equation}\label{eq-rfusup}
 \begin{aligned}
   &\sup_{0 \le t < T^*}R(t)\le C(T^*);\\
   &\sup_{0 \le t < T^*}\left|f(t)\right|_{H_{\omega}^1}\le C(T^*);\\
   &\sup_{0 \le t < T^*} \left|\bu(t)\right|_{H^2}
      \le C(T^*).
 \end{aligned}
\end{equation}
In term of the continuity of $f(t)$, $\bu(t)$ and $R(t)$, we can define
\begin{equation}\label{eq-rfulim}
 \begin{aligned}
   &f(T^*)= \lim_{t \to T^*-}f(t) \quad \text{in $H_{\omega}^1(\bbr^3\times \bbr^3)$};\\
   &\bu(T^*)=\lim_{t \to T^*-}\bu(t) \quad \text{in $H^2(\bbr^3)$}.
 \end{aligned}
\end{equation}
From $\eqref{eq-rfusup}_1$, we know $R(T^*)\le C(T^*)$. Thus, we can take $\Big(f(T^*), \bu(T^*)\Big)$ as an initial datum, and use Proposition \ref{prop-loc-exist} to extend the life span beyond $T^*$. Therefore, $T^*=\infty$, i.e., the system \eqref{eq-cs-s}-\eqref{eq-sys-inidata} admits global-in-time strong solutions. The uniqueness of strong solutions can be proved in the same way as in Proposition \ref{prop-loc-exist}. This completes the proof. $\hfill \square$


\begin{thebibliography}{10}

\bibitem{Bae2012}
H.-O. Bae, Y.-P. Choi, S.-Y. Ha, and M.-J. Kang.
\newblock Time-asymptotic interaction of flocking particles and an
  incompressible viscous fluid.
\newblock {\em Nonlinearity}, 25(4):1155--1177, 2012.

\bibitem{bae2014asymptotic}
H.-O. Bae, Y.-P. Choi, S.-Y. Ha, and M.-J. Kang.
\newblock {Asymptotic flocking dynamics of Cucker--Smale particles immersed in
  compressible fluids}.
\newblock {\em Discrete Contin. Dyn. Syst.}, 34(11):4419--4458, 2014.

\bibitem{bae2014global}
H.-O. Bae, Y.-P. Choi, S.-Y. Ha, and M.-J. Kang.
\newblock {Global existence of strong solution for the
  Cucker--Smale--Navier--Stokes system}.
\newblock {\em J. Differential Equations}, 257(6):2225--2255, 2014.

\bibitem{bae2016global}
H.-O. Bae, Y.-P. Choi, S.-Y. Ha, and M.-J. Kang.
\newblock {Global existence of strong solutions to the Cucker--Smale--Stokes
  system}.
\newblock {\em J. Math. Fluid Mech.}, 18(2):381--396, 2016.

\bibitem{Canizo2011}
J.~A. Canizo, J.~A. Carrillo, and J.~Rosado.
\newblock A well-posedness theory in measures for some kinetic models of
  collective motion.
\newblock {\em Math. Models Methods Appl. Sci.}, 21(03):515--539, 2011.

\bibitem{Carrillo2010}
J.~A. Carrillo, M.~Fornasier, J.~Rosado, and G.~Toscani.
\newblock {Asymptotic flocking dynamics for the kinetic Cucker--Smale model}.
\newblock {\em SIAM J. Math. Anal.}, 42(1):218--236, 2010.

\bibitem{carrillo2010particle}
J.~A. Carrillo, M.~Fornasier, G.~Toscani, and F.~Vecil.
\newblock Particle, kinetic, and hydrodynamic models of swarming.
\newblock In {\em Mathematical Modeling of Collective Behavior in
  Socio-economic and Life Sciences}, pages 297--336. Springer, 2010.

\bibitem{choi2017emergent}
Y.-P. Choi, S.-Y. Ha, and Z.~Li.
\newblock {Emergent dynamics of the Cucker--Smale flocking model and its
  variants}.
\newblock In {\em Active Particles, Volume 1}, pages 299--331. Springer, 2017.

\bibitem{Cucker2007}
F.~Cucker and S.~Smale.
\newblock Emergent behavior in flocks.
\newblock {\em IEEE Trans. Automat. Control}, 52(5):852--862, 2007.

\bibitem{duan2010kinetic}
R.~Duan, M.~Fornasier, and G.~Toscani.
\newblock A kinetic flocking model with diffusion.
\newblock {\em Commun. Math. Phys.}, 300(1):95--145, 2010.

\bibitem{galdi2011introduction}
G.~Galdi.
\newblock {\em An Introduction to the Mathematical Theory of the Navier--Stokes
  Equations: Steady-state Problems}.
\newblock Springer Science \& Business Media, 2011.

\bibitem{giga1991abstract}
Y.~Giga and H.~Sohr.
\newblock {Abstract Lp estimates for the Cauchy problem with applications to
  the Navier--Stokes equations in exterior domains}.
\newblock {\em J. Funct. Anal.}, 102(1):72--94, 1991.

\bibitem{ha2014global}
S.-Y. Ha, F.~Huang, and Y.~Wang.
\newblock {A global unique solvability of entropic weak solution to the
  one-dimensional pressureless Euler system with a flocking dissipation}.
\newblock {\em J. Differential Equations}, 257(5):1333--1371, 2014.

\bibitem{Ha2014}
S.-Y. Ha, M.-J. Kang, and B.~Kwon.
\newblock {A hydrodynamic model for the interaction of Cucker--Smale particles
  and incompressible fluid}.
\newblock {\em Math. Models Methods Appl. Sci.}, 24(11):2311--2359, 2014.

\bibitem{ha2015emergent}
S.-Y. Ha, M.-J. Kang, and B.~Kwon.
\newblock {Emergent dynamics for the hydrodynamic Cucker--Smale system in a
  moving domain}.
\newblock {\em SIAM J. Math. Anal.}, 47(5):3813--3831, 2015.

\bibitem{Ha2009}
S.-Y. Ha and J.-G. Liu.
\newblock {A simple proof of the Cucker--Smale flocking dynamics and mean-field
  limit}.
\newblock {\em Commun. Math. Sci.}, 7(2):297--325, 2009.

\bibitem{jin2015well}
C.~Jin.
\newblock {Well posedness for pressureless Euler system with a flocking
  dissipation in Wasserstein space}.
\newblock {\em Nonlinear Anal.}, 128:412--422, 2015.

\bibitem{Jin}
C.~Jin.
\newblock {Well posedness for pressureless Euler system with a flocking
  dissipation}.
\newblock {\em Acta Math. Sci. Ser. B Engl. Ed.}, 36(5):1262--1284, 2016.

\bibitem{Jin2018}
C.~Jin.
\newblock {Well-posedness of weak and strong solutions to the kinetic
  Cucker--Smale model}.
\newblock {\em J. Differential Equations}, 264(3):1581--1612, 2018.

\bibitem{jin2019local}
C.~Jin.
\newblock {The local existence and blowup criterion for strong solutions to the
  kinetic Cucker--Smale model coupled with the compressible Navier--Stokes
  equations}.
\newblock {\em Nonlinear Anal.- Real.}, 49:217--249, 2019.

\bibitem{sohr2012navier}
H.~Sohr.
\newblock {\em The Navier--Stokes Equations: An Elementary Functional Analytic
  Approach}.
\newblock Springer Science \& Business Media, 2012.

\end{thebibliography}

\end{document}